\documentclass[10pt]{amsart}

\usepackage{geometry}
\usepackage{amssymb}
\usepackage{verbatim, color}

\newtheorem{lm}{Lemma}[section]
\newtheorem{teo}[lm]{Theorem}
\newtheorem{prop}[lm]{Proposition}
\newtheorem{coro}[lm]{Corollary}

\theoremstyle{definition}
\newtheorem{oss}[lm]{Remark}
\newtheorem{defi}[lm]{Definition}
\newtheorem*{ack}{Acknowledgments}

\usepackage{color}

\usepackage[colorlinks=true,urlcolor=blue, citecolor=blue,linkcolor=blue,
linktocpage,pdfpagelabels, bookmarksnumbered,bookmarksopen]{hyperref}
\usepackage[hyperpageref]{backref}

\numberwithin{equation}{section}


\title[Long-time behavior]{Long-time behavior\\ for the porous medium equation\\ with small initial energy}
\author[Brasco]{Lorenzo Brasco}
\address[L.\ Brasco]{Dipartimento di Matematica e Informatica
\newline\indent
Universit\`a degli Studi di Ferrara
\newline\indent
Via Machiavelli 30, 44121 Ferrara, Italy}
\email{lorenzo.brasco@unife.it}

\author[Volzone]{Bruno Volzone}

\address[B.\ Volzone]{Dipartimento di Scienze e Tecnologie
\newline\indent
Universit\`a degli Studi di Napoli ``Parthenope''
\newline\indent
Centro Direzionale Isola C4, 80143 Napoli, Italy}
\email{bruno.volzone@uniparthenope.it}

\begin{document}

\begin{abstract}
We study the long-time behavior for the solution of the Porous Medium Equation in an open bounded connected set, with smooth boundary. Homogeneous Dirichlet boundary conditions are considered. We prove that if the initial datum has sufficiently small energy, then the solution converges to a nontrivial constant sign solution of a sublinear Lane-Emden equation, once suitably rescaled. We point out that the initial datum is allowed to be sign-changing.
\par
We also give a sufficient energetic criterion on the initial datum, which permits to decide whether convergence takes place towards the positive solution or to the negative one.
\end{abstract}

\maketitle

\begin{center}
\begin{minipage}{11cm}
\small
\tableofcontents
\end{minipage}
\end{center}

\section{Introduction}

\subsection{Overview}
Let us set $\Phi(s)=|s|^{m-1}\,s$, for an exponent $m>1$.
On a given open bounded set $\Omega\subset\mathbb{R}^N$, we consider the following initial boundary value problem for the {\it Porous Medium Equation} (PME for short)
\begin{equation}\label{intromainprob}
\left\{\begin{array}{rcll}
\partial_t u&=&\Delta \Phi(u), & \mbox{ in } Q:=\Omega\times (0,+\infty),\\
u&=&0, & \mbox{ on } \Sigma:=\partial\Omega\times(0,+\infty),\\
u(\cdot,0)&=&u_0,& \mbox{ in } \Omega.
\end{array}
\right.
\end{equation}
We are concerned in this paper with the long-time behavior of the solution $u$ to \eqref{intromainprob}. We point out that in this paper {\it we do not take any sign assumption on $u_0$}.
\vskip.2cm\noindent
By looking for stationary solutions of the PME, i.e. solutions of the form $u(x,t)=X(x)\,T(t)$, it is not difficult to guess that a solution $u$ to \eqref{intromainprob} should behave like
\begin{equation}
\label{LTB}
u(x,t)\sim t^{-\frac{1}{m-1}},\qquad \mbox{ for } t\nearrow +\infty.
\end{equation}
However, differently from the case $m=1$ (i.e. the heat equation), the equation is now nonlinear and thus this formal argument does not lead to a rigorous proof. In other words, for $m>1$ the solution $u$ can not be written as a superposition of stationary solutions, differently from the case $m=1$.
\par
In the case of a nontrivial initial datum $u_0\ge 0$, the long-time behavior \eqref{LTB} has been first rigourously proved by Aronson and Peletier in \cite[Theorem 3]{AP}, under some smoothness assumptions on $\partial\Omega$ and $u_0$. More precisely, they proved that
\begin{equation}
\label{AP}
\lim_{t\to+\infty} \Big\|t^\frac{1}{m-1}\,u(\cdot,t)-U\Big\|_{L^\infty(\Omega)}=0,
\end{equation}
where $U$ is such that $\Phi(U)\in W^{1,2}_0(\Omega)$ is the unique positive weak solution of the {\it sublinear Lane-Emden equation}
\begin{equation}
\label{LEintro}
-\Delta \psi=\alpha\,|\psi|^{q-2}\,\psi,\quad \mbox{ in }\Omega,\qquad \mbox{ with } \alpha=\frac{1}{m-1}\ \mbox{Â and }\ q=\frac{m+1}{m}.
\end{equation}
Actually, the result in \cite{AP} is more precise, as it comes with a quantitative estimate on the rate of convergence with respect to the relative error. Later on, \cite[Theorem 3]{AP} has been generalized by Vazquez in \cite[Theorem 1.1]{Va}, by means of a simpler proof, based on maximum principles and on the observation that $u$ is monotone increasing in time, up to a suitable time scaling.
\par
After the pioneering result by Aronson and Peletier, a lot of efforts has been devoted to generalize in various directions their result for positive initial data. Without any attempt of completeness, we wish to mention for example: the paper \cite{BP} where more general nonlinearities $\Phi(u)$ are allowed; a handful of recent references \cite{BFV, BSV, GMP}, dealing with nonlocal versions of the PME; the reference \cite{GMV} studying the long-time behavior for the PME on more general geometries (noncompact Riemannian manifolds).
\vskip.2cm\noindent
On the contrary, the case when the initial datum $u_0$ is {\it sign-changing}, i.e.
\[
u_0^+:=\max\{u_0,0\}\not\equiv 0\qquad \mbox{ and }\qquad u_0^-:=\min\{u_0,0\}\not\equiv 0,
\]
is less investigated and more difficult to handle.
We take the occasion to point out that, even if usually the constant sign case is considered to be more appropriate for applications, the sign-changing case has its own physical relevance.
We refer for example to \cite{DV}, which contains a model from Hydrology leading to the study of equations of the type \eqref{intromainprob} with sign-changing solutions.
\par
Despite some theoretical studies, see for example \cite{HKB, Ma} and the references therein, in this case the only convergence result we are aware of is \cite[Theorem 2.6]{Va} by Vazquez, which deals with the one-dimensional case $N=1$. In this case, it is possible to prove that for every $u_0\in L^1(\Omega)$, we still have \eqref{AP} and again $U$ is such that $\Phi(U)$ weakly solves \eqref{LEintro}.
\par
The one-dimensional case is quite special: the result \cite[Theorem 2.6]{Va} heavily relies on the fact that solutions to \eqref{LEintro} on an interval $\Omega=(a,b)$ are completely classified (see \cite[Lemma 2.2]{Va}). In higher dimensions, the situation is much more complicated, even for simple geometries and such a classification result is not available. Moreover, the proof of \cite{Va} exploits the stabilization result of \cite[Theorem 18]{ACP} by Aronson, Crandall and Peletier, which does not seem easy to be generalized for $N\ge 2$, see the comment in \cite[page 1022]{ACP}.
\par
In any case, it is important to point out that \cite[Theorem 2.6]{Va} does not exclude the possibility that $U\equiv 0$ when $u_0$ is sign-changing. This leaves open the question of the optimality of the estimate \eqref{LTB}
in the case of a sign-changing initial datum $u_0$. We refer to \cite[Section 4.2]{Va} for further comments on this point.
\begin{oss}
Even if this is not the subject of this paper, for completeness we recall that the long-time behavior of solution to \eqref{intromainprob} has been widely investigated also in the case $0<m<1$. In this case, the relevant parabolic equation is known as {\it fast diffusion equation}. For $0<m<1$, the long-time behavior is more complicated already in the constant sign case $u_0\ge 0$. Indeed, two major difficulties now arise: at first, the solution becomes identically zero after a certain {\it extinction time} $T^*$.
Secondly, in the stationary equation \eqref{LEintro} the exponent $q=(m+1)/m$ is now {\it larger than $2$}, i.e. the term $|\psi|^{q-2}\,\psi$ is {\it superlinear}. Then it is well-known that equation \eqref{LEintro} may have multiple solutions, already in the constant-sign case: indeed, the multiplicity of positive solutions depends very much on the geometry of the underlying set $\Omega$ (see for example \cite[Corollary]{Da} and \cite[Example 4.7]{BraFra} for some non-convex starshaped sets having multiple positive solutions).
\par
For more details and some partial results about the long-time behavior for $0<m<1$, we refer the reader to the pioneering result of \cite[Theorem 2]{BH}, as well as to the recent paper \cite{BF} and the references therein.
\end{oss}

\subsection{Main results}
In order to present our results, we need to settle some definitions.
For $1<q<2$ and $\alpha>0$, we define the following energy functional
\[
\mathfrak{F}_{q,\alpha}(\varphi)=\frac{1}{2}\,\int_\Omega |\nabla \varphi|^2\,dx-\frac{\alpha}{q}\,\int_\Omega |\varphi|^q\,dx,\qquad \mbox{ for every } \varphi\in W^{1,2}_0(\Omega).
\]
We say that $\lambda\in \mathbb{R}$ is a {\it critical level} for $\mathfrak{F}_{q,\alpha}$ provided there exists a weak solution $u\in W^{1,2}_0(\Omega)$ of
\[
-\Delta u=\alpha\,|u|^{q-2}\,u,\qquad \mbox{ in }\Omega,
\]
with
\[
\mathfrak{F}_{q,\alpha}(u)=\lambda.
\]
We refer the reader to Section \ref{sec:energyfunctional} below for some basic properties of the functional $\mathfrak{F}_{q,\alpha}$.
\par
In particular, the following two critical levels of $\mathfrak{F}_{q,\alpha}$ will play a major role: the {\it ground state level}
\[
\Lambda_1:=\min_{\varphi\in W^{1,2}_0(\Omega)} \mathfrak{F}_{q,\alpha}(\varphi),
\]
and the {\it first excited state level}
\[
\Lambda_2:=\inf\Big\{\lambda>\Lambda_1\, :\, \lambda \mbox{ is a critical value of }\mathfrak{F}_{q,\alpha}\Big\}.
\]
We then have the following convergence result, which is the main outcome of our paper.
\begin{teo}
\label{teo:main}
Let $m>1$ and let $\Omega\subset\mathbb{R}^N$ be an open bounded connected set, with $C^1$ boundary. Let us assume that the initial datum $u_0\in L^{m+1}(\Omega)$ satisfies:
\begin{itemize}
\item[\it (i)] $\Phi(u_0)\in W^{1,2}_0(\Omega)$;
\vskip.2cm
\item[\it (ii)] if we set $q=(m+1)/m$ and $\alpha=1/(m-1)$, then
\begin{equation}
\label{condizione}
\mathfrak{F}_{q,\alpha}(\Phi(u_0))<\Lambda_2.
\end{equation}
\end{itemize}
If $u$ is the unique weak solution to \eqref{intromainprob}, then there exists $U\in C(\overline{\Omega})\setminus\{0\}$ such that
\[
\lim_{t\to+\infty} \|t^\alpha u(\cdot,t)-U\|_{L^\infty(\Omega)}=0.
\]
Moreover, $\Phi(U)\in \{w,-w\}$ where $w$ is the unique positive minimizer of $\mathfrak{F}_{q,\alpha}$ over $W^{1,2}_0(\Omega)$.
\end{teo}
Observe that, since the PME has a local character, its dynamics evolve independently on each connected component of a set. Hence, the connectedness of $\Omega$ is a reasonable assumption.
\par
To the best of our knowledge, this is the first convergence result for the PME, with no sign assumptions on the initial datum and for dimensions $N\ge 2$. We will explain in the next subsection the importance of the condition \eqref{condizione}. We first make a comment on the assumptions on the set.
\begin{oss}[Assumptions on $\Omega$]
For a general open bounded set $\Omega\subset\mathbb{R}^N$ it may happen that $\Lambda_2=\Lambda_1$. However, the $C^1$ assumption on $\Omega$ assures that \[
\Lambda_1<\Lambda_2,
\]
see Proposition \ref{prop:dasolo} below, whose proof crucially exploits the recent result of \cite[Theorem A]{BDF}. Thus, our assumption \eqref{condizione} is meaningful. Moreover, it is easily seen that \eqref{condizione} is compatible with taking sign-changing initial data, see Lemma \ref{lm:nontrivial} below.
\par
Actually, the regularity requirement on $\Omega$ could be weakened, we refer to Remark \ref{oss:BDF} below.
\par
We also remark that Theorem \ref{teo:main} still holds for $\Omega$ having a {\it finite} number of connected components $\Omega_1,\dots,\Omega_k$. The only change in the statement is that now the functional $\mathfrak{F}_{q,\alpha}$ has $2^k$ minimizers over $W^{1,2}_0(\Omega)$. See Remarks \ref{oss:disconnected} and \ref{oss:BDF} below.
\end{oss}
The result of Theorem \ref{teo:main} prescribes convergence to a function $U$ such that $\Phi(U)$ is a global minimizer of $\mathfrak{F}_{q,\alpha}$, up to the scaling factor $t^\alpha$. Since for a connected set there are two such minimizers, i.\,e. $w$ and $-w$, it would be interesting to have a sufficient condition to decide whether the stabilization point of $t^\alpha\,u(\cdot,t)$ is $\Phi^{-1}(w)$ or $\Phi^{-1}(-w)$. Then the previous result has to be coupled with the following one, giving a sufficient condition that ensures convergence\footnote{Of course, this gives in a dual way a sufficient condition to get stabilization towards $\Phi^{-1}(-w)$. It is sufficient to exchange the roles of $u_0^+$ and $u_0^-$.} to $\Phi^{-1}(w)$.
\begin{prop}[Selection criterion]
\label{prop:seleziona}
Under the assumptions of Theorem \ref{teo:main}, we further suppose that the initial datum $u_0$ satisfies one of the following properties:
\begin{enumerate}
\item[A.] either
\begin{equation}
\label{vaialpositivo}
\mathfrak{F}_{q,\alpha}\left(\Phi(u_0^-)\right)\ge 0;
\end{equation}
\item[B.] or
\begin{equation}
\label{rivaialpositivo}
\mathfrak{F}_{q,\alpha}\left(\Phi(u_0^-)\right)<0\qquad\mbox{ and }\qquad \mathfrak{F}_{q,\alpha}\left(\Phi(u_0^+)\right)<\Lambda_2.
\end{equation}
\end{enumerate}
Then we have
\[
\lim_{t\to+\infty} \|t^\alpha u(t,\cdot)-\Phi^{-1}(w)\|_{L^\infty(\Omega)}=0.
\]
\end{prop}
\begin{oss}
One could wonder whether the class of sign-changing data that comply with the additional conditions A. or B. is empty or not. Actually, we can assure that it is always possible to construct initial data $u_0$ with $u_0^+\not\equiv 0$, $u_0^-\not\equiv0$ and such that condition A. is satisfied (see Lemma \ref{lm:nontrivial} below).
\par
Our ``selection criterion'' does not cover the case where
\begin{equation}
\mathfrak{F}_{q,\alpha}\left(\Phi(u_0^-)\right)<0\qquad\mbox{ and }\qquad \mathfrak{F}_{q,\alpha}\left(\Phi(u_0^+)\right)\ge \Lambda_2.\label{condizameta}
\end{equation}
In this situation our proof does not work.
In any case, we suspect that assuming \eqref{condizameta} the dynamics are in general much more complicated: it may be that other properties of $u_0$ influences the long-time stabilization, apart for the energy of $\Phi(u_0^+)$ and $\Phi(u_0^-)$.
\end{oss}

\subsection{Methods of proof}

The idea behind the proof of Theorem \ref{teo:main} is quite simple to explain. At first, as it is now customary, we perform a time scaling
\[
v(x,t)=\text{e}^{\alpha\, t}\,u(x,e^{t}-1),\qquad \mbox{ where } \alpha=\frac{1}{m-1}.
\]
This new function $v$ solves
\begin{equation}\label{rescaledfirst}
\left\{\begin{array}{rcll}
\partial_t v&=&\Delta \Phi(v)+\alpha\,v, & \mbox{ in } Q,\\
v&=&0, & \mbox{ on } \Sigma,\\
v(\cdot,0)&=&u_0,& \mbox{ in } \Omega.
\end{array}
\right.
\end{equation}
Then one can formally observe that $\mathfrak{F}_{q,\alpha}$ is a Lyapunov functional for this dynamical system, with the aforementioned choices of $\alpha$ and $q$. In other words, the function
\[
t\mapsto \mathfrak{F}_{q,\alpha}(\Phi(v(\cdot,t))),
\]
is non-increasing, as time goes by.
Thus, if we start with an initial datum which satisfies \eqref{condizione}, we will constantly stay below the threshold $\Lambda_2$. Then, in a nutshell, we use the following three ingredients:
\begin{itemize}
\item in the limit as $t$ goes to $+\infty$, we have to stabilize towards a function $U$ such that $\Phi(U)$ is a critical point of $\mathfrak{F}_{q,\alpha}$;
\vskip.2cm
\item the {\it $\omega-$limit set}, i.e. the collection of all possible limit points of $v(\cdot,t)$, is a connected set;
\vskip.2cm
\item below the level $\Lambda_2$, the critical points of $\mathfrak{F}_{q,\alpha}$ form a disconnected set.
\end{itemize}
These three points show that we must have convergence of $v$ for $t$ diverging to $+\infty$ to a certain profile. This in turn implies the desired convergence of the original solution $u$.
\par
As simple as it seems, we have to pay attention to a detail: as observed by Langlais and Phillips in \cite[Remark 1.2]{LP}, proving that $\mathfrak{F}_{q,\alpha}$ decreases in time seems to require some strong compactness of $\nabla \Phi(v(\cdot,t))$ in the $W^{1,2}_0(\Omega)$ topology. In dimension $N=1$, this approach has been successfully employed in the aforementioned result \cite[Theorem 18]{ACP} by Aronson, Crandall and Peletier. However, their proof exploits a sophisticated second order estimate by Benilan and Crandall contained in \cite{BC}. As already said, for $N\ge 2$ it seems quite complicated to extend this approach, since the estimate of \cite{BC} is not enough to conclude.
\par
Thus, we decided here to take a slightly different path, as in \cite{LP}. In particular, we just rely on the fact that
\begin{equation}
\label{barriera}
\mathfrak{F}_{q,\alpha}(\Phi(v(\cdot,t)))\le \mathfrak{F}_{q,\alpha}(\Phi(u_0)),\qquad \mbox{ for every }t>0.
\end{equation}
This is a consequence of the so-called {\it entropy--entropy dissipation inequality}. The (apparently) weaker information \eqref{barriera} is actually all that we need: it guarantees that the stabilization takes place at an energy level which stays below $\Lambda_2$, as needed. Then the scheme of proof highlighted above can be successfully applied.
\par
In any case, incidentally from \eqref{barriera} we immediately have that $\mathfrak{F}_{q,\alpha}$ is indeed monotone decreasing along the solution. Taking $t>s>0$, since $v(x,t)$ solves \eqref{rescaledfirst} in the time interval $[s,+\infty)$ with initial datum $v(s)$, we obtain
\[
\mathfrak{F}_{q,\alpha}(\Phi(v(\cdot,t)))\le\mathfrak{F}_{q,\alpha}(\Phi(v(\cdot,s))),
\]
just by replacing $u_0$ with $v(\cdot,s)$ at the right-hand side of \eqref{barriera}.
\begin{oss}
In the limit case where the initial datum $u_0$ is such that $\mathfrak{F}_{q,\alpha}(\Phi(u_0))=\Lambda_2$, in general the conclusion of Theorem \ref{teo:main} does not hold. For example, by taking $u_0\in L^{m+1}(\Omega)$ such that
\[
-\Delta \Phi(u_0)=\alpha\,|\Phi(u_0)|^{q-2}\,\Phi(u_0)\qquad \mbox{ and }\qquad \mathfrak{F}_{q,\alpha}(\Phi(u_0))=\Lambda_2,
\]
we get that
\[
u(x,t)=(1+t)^{-\alpha}\,u_0(x),
\]
is the unique solution of \eqref{intromainprob}. In this case, we have of course that $t^\alpha\,u(\cdot,t)$ converges to $u_0$, as $t$ goes to $+\infty$. We observe that such a choice of $u_0$ is feasible, since $\Lambda_2$ is a critical level for $\mathfrak{F}_{q,\alpha}$. Also notice that the function $u_0$ is indeed sign-changing, by Proposition \ref{prop:constantsign} below.
\par
Moreover, in the case $\mathfrak{F}_{q,\alpha}(\Phi(u_0))\ge \Lambda_2$, the scheme of proof previously presented may fail to work, because one can not exclude a priori (differently from the one-dimensional case) that a critical level of $\mathfrak{F}_{q,\alpha}$ has a non-trivial topology. This happens for example when $N=2$ and $\Omega$ is a disk: in this case, we have a continuum of critical points of $\mathfrak{F}_{q,\alpha}$ homemorphic to $\mathbb{S}^1$, all having the same energy and consisting of the solutions of the corresponding Lane-Emden equation having a diameter as a nodal line.
\end{oss}
As for the ``selection criterion'' of Proposition \ref{prop:seleziona}, this is based on the existence of a critical level $\Lambda^*$ for $\mathfrak{F}_{q,\alpha}$, having a mountain pass nature. More precisely, this is defined by
\[
\Lambda^*:=\inf_{\gamma\in \Gamma}\max_{\varphi\in \mathrm{Im}(\gamma)} \mathfrak{F}_{q,\alpha}(\varphi),
\]
where
\[
\Gamma=\Big\{\gamma\in C([0,1];W^{1,2}_0(\Omega))\, :\, \gamma(0)=w,\, \gamma(1)=-w\Big\}.
\]
Under the assumptions of Proposition \ref{prop:seleziona}, if $\Phi(v)$ would stabilize towards $-w$, then one could construct a continuous path connecting $w$ and $-w$, whose energy constantly stays strictly below $\Lambda_2\le \Lambda^*$. This would violate the definition of $\Lambda^*$. The construction of this path exploits again \eqref{barriera} and a ``hidden convexity'' property of $\mathfrak{F}_{q,\alpha}$ (which {\it is not} convex in the usual sense), see Lemma \ref{lm:mistero}.
\par
Here as well, we need to pay attention to a technical detail: as it is typical in minimax theories, the set $\Gamma$ is made of paths which are continuous in the {\it strong topology} of $W^{1,2}_0(\Omega)$. This poses some troubles, since we want to use $t\mapsto \Phi(v(\cdot,t))$ to construct a piece of the aforementioned path. This would need an extra regularity estimate granting the continuity of this path in the strong $W^{1,2}_0(\Omega)$ topology.
\par
We circumvent this problem, by proving that actually the continuity requirement in $\Gamma$ can be considerably relaxed. Actually, it is sufficient to have continuity in the $L^1$ strong topology, see Corollary \ref{coro:weakest}. We believe this result to be interesting in itself, its proof being inspired to that of \cite[Theorem 3.3]{CD}.

\subsection{Plan of the paper}
We start by setting some of the definitions and basic results on the PME, in Section \ref{sec:2}. This section also contains some basic facts on the energy functional $\mathfrak{F}_{q,\alpha}$. We quickly pursue the investigation on $\mathfrak{F}_{q,\alpha}$ in Sections \ref{sec:3} and \ref{sec:4}, dealing with the first and second critical level, respectively. Section \ref{sec:4} Â also contains a detailed discussion of the mountain pass level $\Lambda^*$. With Section \ref{sec:5} we enter into the core of the paper: here the main result is Proposition \ref{prop:stimecruciali}, containing some crucial a priori estimates for the rescaled solution $v$. Finally, in Section \ref{sec:6} we prove our main results.
\par
The paper is complemented by two final appendices, which contribute to make the paper as self-contained as possible.

\begin{ack}
L.\,B. was financially supported by the Italian grant FFABR {\it Fondo Per il Finanziamento delle attivit\`a di base}. B.\,V. was partially supported by GNAMPA of the INdAM (Istituto Nazionale di Alta Matematica). We would like to thank Ugo Gianazza, Matteo Muratori, Marco Squassina and Juan Luis V\'azquez  for the fruitful discussions and valuable suggestions.
\par
 Finally, we wish to thank an anonymous referee for his thorough reading and for coming with many comments and suggestions, which lead to an improvement of the paper.
\end{ack}

\section{Preliminaries}
\label{sec:2}

\subsection{Notation}

For an open bounded set $\Omega\subset\mathbb{R}^N$, we indicate by $W^{1,2}_0(\Omega)$ the closure of $C^\infty_0(\Omega)$ with respect to the norm
\[
\|\varphi\|_{W^{1,2}(\Omega)}=\|\varphi\|_{L^2(\Omega)}+\|\nabla \varphi\|_{L^2(\Omega;\mathbb{R}^N)}.
\]
We will endow $W^{1,2}_0(\Omega)$ with the equivalent norm
\[
\|\varphi\|_{W^{1,2}_0(\Omega)}:=\|\nabla \varphi\|_{L^2(\Omega)},\qquad \mbox{ for every } \varphi\in W^{1,2}_0(\Omega).
\]
Equivalence of these norms follows from Poincar\'e inequality, i.e. the fact that
\[
\lambda_1(\Omega)=\min_{u\in W^{1,2}_0(\Omega)} \left\{\int_\Omega |\nabla u|^2\,dx\, :\, \|u\|_{L^2(\Omega)}=1\right\}>0.
\]
For $1<q<2$, we also define
\[
\lambda_1(\Omega;q)=\min_{u\in W^{1,2}_0(\Omega)} \left\{\int_\Omega |\nabla u|^2\,dx\, :\, \|u\|_{L^q(\Omega)}=1\right\},
\]
i.\,e. the sharp constant in the Sobolev-Poincar\'e inequality
\begin{equation}
\label{sobolevpoincare}
c\,\left(\int_\Omega |\varphi|^q\,dx\right)^\frac{2}{q}\le \int_\Omega |\nabla \varphi|^2\,dx,\qquad \mbox{ for every }\varphi\in W^{1,2}_0(\Omega).
\end{equation}
By combining interpolation in Lebesgue spaces and the Poincar\'e inequality, for $1<q<2$ we also get the following interpolation inequality
\begin{equation}
\label{gagliardo}
\left(\int_\Omega |\varphi|^q\,dx\right)^\frac{2}{q}\le \Big(\lambda_1(\Omega)\Big)^{\vartheta-1}\,\left(\int_\Omega |\varphi|\,dx\right)^{2\,\vartheta}\,\left(\int_\Omega |\nabla \varphi|^2\,dx\right)^{1-\vartheta},\ \mbox{ for every }\varphi\in W^{1,2}_0(\Omega).
\end{equation}
Here the exponent $\vartheta$ is dictated by scale invariance, thus it is given by
\[
\vartheta=\frac{2}{q}-1.
\]

\subsection{Weak solutions to the PME}

In what follows, we set for every $m>1$
\[
\Phi(s)=|s|^{m-1}\,s,\qquad \mbox{ for every } s\in\mathbb{R}.
\]
In this section, we will indicate by $\Omega\subset\mathbb{R}^N$ any open bounded set, without any further assumption.
We recall some basic results on the homogeneous Dirichlet problem (HDP for short) for the \emph{signed} PME
\begin{equation}\label{mainprob}
\left\{\begin{array}{rcll}
\partial_{t} u&=&\Delta \Phi(u), & \mbox{ in } Q:=\Omega\times (0,+\infty),\\
u&=&0, & \mbox{ on } \Sigma:=\partial\Omega\times(0,+\infty),\\
u(\cdot,0)&=&u_0,& \mbox{ in } \Omega.
\end{array}
\right.
\end{equation}
In particular, we recall here the basic definitions of weak solution and some related properties, see \cite[Definition 6.5]{VaBook}. The assumption on the initial datum is justified by our settings. We set $Q_{T}=\Omega\times(0,T)$, for $T>0$.
\begin{defi}\label{weaksolutiPME}
Let $u_{0}\in L^{m+1}(\Omega)$ be such that $\Phi(u_{0})\in W_{0}^{1,2}(\Omega)$. A function $u\in C([0,+\infty); L^{m+1}(\Omega))$ is said to be a {\it weak solution} to problem \eqref{mainprob} if:
\begin{itemize}
\item[(\emph{i})] $\Phi(u)\in L^{2}_{\rm loc}([0,+\infty);W_{0}^{1,2}(\Omega))$;
\vskip.2cm
\item[(\emph{ii})] $u$ satisfies the identity
\[
\iint_{Q}\Big(\langle\nabla\Phi(u),\nabla\eta\rangle-u\,\partial_t \eta\Big)\,dx\,dt=0,
\]
for any test function $\eta\in C_{0}^\infty(Q)$;
\vskip.2cm
\item[(\emph{iii})] $u(\cdot,0)=u_{0}$, in the sense that
\[
\lim_{t\to 0^+} \|u(\cdot,t)-u_0\|_{L^{m+1}(\Omega)}=0.
\]
\end{itemize}
\end{defi}
We recall the following existence and uniqueness result, which can be found in \cite[Theorem 6.12]{VaBook}. In what follows, we use the notation
\begin{equation}
\label{g}
g(s)=|s|^\frac{m-1}{2}\,s,\qquad \mbox{ for } s\in\mathbb{R}.
\end{equation}
\begin{teo}\label{existence}
For all $u_{0}\in L^{m+1}(\Omega)$ such that $\Phi(u_{0})\in W_{0}^{1,2}(\Omega)$, there exists a unique weak solution $u$ to \eqref{mainprob}. Moreover, for all $T>0$ we have
\[
u\in L^{\infty}((0,T);L^{m+1}(\Omega)),
\]
and the energy inequalities hold
\[
\iint_{Q_{T}}|\nabla \Phi(u)|^{2}\,dx\,dt+\frac{1}{m+1}\,\int_{\Omega}|u(x,T)|^{m+1}\,dx\leq \frac{1}{m+1}\int_{\Omega}|u_{0}|^{m+1}\,dx,
\]
\[
\int_{\Omega}|\nabla\Phi(u(x,T))|^{2}\,dx+\frac{8\,m}{(m+1)^{2}}\,\iint_{Q_{T}}|\partial_{t}g(u)|^{2}\,dx\,dt\leq\int_{\Omega}|\nabla\Phi(u_{0})|^{2}\,dx.
\]
In particular, we get that $\Phi(u)\in L^{2}([0,+\infty);W_{0}^{1,2}(\Omega))$.
\end{teo}
The second part of Theorem \ref{existence} follows from \cite[Theorem 6.13, Theorem 5.7]{VaBook}.
\vskip.2cm

\subsection{Energy functional}
\label{sec:energyfunctional}

We still indicate by $\Omega\subset\mathbb{R}^N$ any open bounded set, without any further assumption, unless explicitly stated.
Let $1<q<2$ and $\alpha>0$, we introduce the functional defined over $W^{1,2}_0(\Omega)$ by
\begin{equation}
\label{functionalF}
\mathfrak{F}_{q,\alpha}(\varphi)=\frac{1}{2}\,\int_\Omega |\nabla \varphi|^2\,dx-\frac{\alpha}{q}\,\int_\Omega |\varphi|^q\,dx.
\end{equation}
Any critical point $u\in W^{1,2}_0(\Omega)$ of this functional is a weak solution of the sublinear Lane-Emden equation
\begin{equation}
\label{LE}
-\Delta u=\alpha\,|u|^{q-2}\,u,\qquad \mbox{ in }\Omega.
\end{equation}
In other words, it verifies
 \begin{equation}
\label{LEweak}
\int_\Omega \langle\nabla u,\nabla \varphi\rangle\,dx=\alpha\,\int_\Omega |u|^{q-2}\,u\,\varphi\,dx,\qquad \mbox{ for every }\varphi\in W^{1,2}_0(\Omega).
\end{equation}
By taking the test function $\varphi=u$ in \eqref{LEweak}, we get in particular
\[
\int_\Omega |\nabla u|^2\,dx=\alpha\,\int_\Omega |u|^q\,dx.
\]
Thus, the energy of any critical point $u$ can be written as follows
\begin{equation}
\label{ecrit}
\mathfrak{F}_{q,\alpha}(u)=\left(\frac{1}{2}-\frac{1}{q}\right)\int_\Omega |\nabla u|^2\,dx=\left(\frac{\alpha}{2}-\frac{\alpha}{q}\right)\,\int_\Omega |u|^q\,dx.
\end{equation}
We indicate by
\[
\mathrm{Crit}(\mathfrak{F}_{q,\alpha})=\Big\{\lambda\in\mathbb{R}\, :\, \lambda=\mathfrak{F}_{q,\alpha}(u) \mbox{ for some } u\in W^{1,2}_0(\Omega) \mbox{ solving }\eqref{LE}\Big\},
\]
the collection of all critical values of $\mathfrak{F}_{q,\alpha}$.
\begin{lm}
\label{lm:closed}
$\mathrm{Crit}(\mathfrak{F}_{q,\alpha})$ is a closed subset of $(-\infty,0]$.
\end{lm}
\begin{proof}
From formula \eqref{ecrit} we see that $\mathrm{Crit}(\mathfrak{F}_{q,\alpha})\subset(-\infty,0]$. Let us show that this is a closed subset.
Let $\{\lambda_n\}_{n\in\mathbb{N}}$ be a sequence of critical values, such that
\[
\lim_{n\to\infty}\lambda_n=\lambda.
\]
We need to show that $\lambda$ is a critical value, as well. Let $u_n\in W^{1,2}_0(\Omega)$ be a solution of \eqref{LE}, such that
\[
\mathfrak{F}_{q,\alpha}(u_n)=\lambda_n.
\]
By using \eqref{ecrit}, we have
\[
\lambda_n=\mathfrak{F}_{q,\alpha}(u_n)=\left(\frac{1}{2}-\frac{1}{q}\right)\,\int_\Omega |\nabla u_n|^2\,dx.
\]
In particular, there exists a constant $C$ such that
\[
\int_\Omega |\nabla u_n|^2\,dx\le C,\qquad \mbox{ for every }n\in\mathbb{N}.
\]
By the Rellich-Kondra\v{s}hov Theorem and the weak closedness of $W^{1,2}_0(\Omega)$, there exists $u\in W^{1,2}_0(\Omega)$ such that $\{u_n\}_{n\in\mathbb{N}}$ converges (up to a subsequence) to $u$, weakly in $W^{1,2}_0(\Omega)$ and strongly in $L^q(\Omega)$. Thus, for every $\varphi\in W^{1,2}_0(\Omega)$, we can pass to the limit in
\[
\int_\Omega \langle\nabla u_n,\nabla \varphi\rangle\,dx=\alpha\,\int_\Omega |u_n|^{q-2}\,u_n\,\varphi\,dx,
\]
and obtain that $u$ verifies
\[
\int_\Omega \langle\nabla u,\nabla \varphi\rangle\,dx=\alpha\,\int_\Omega |u|^{q-2}\,u\,\varphi\,dx.
\]
This shows that $u$ is a critical point for $\mathfrak{F}_{q,\alpha}$. In order to compute its energy, it is enough to use again \eqref{ecrit}, so to infer
\[
\begin{split}
\lambda=\lim_{n\to\infty}\lambda_n=\lim_{n\to\infty}\mathfrak{F}_{q,\alpha}(u_n)&=\lim_{n\to\infty}\left(\frac{\alpha}{2}-\frac{\alpha}{q}\right)\,\int_\Omega |u_n|^q\,dx\\
&=\left(\frac{\alpha}{2}-\frac{\alpha}{q}\right)\,\int_\Omega |u|^q\,dx=\mathfrak{F}_{q,\alpha}(u).
\end{split}
\]
This shows that $\lambda$ is a critical value, as desired.
\end{proof}
In what follows, for every $u\in W^{1,2}_0(\Omega)$ we denote by $d\,\mathfrak{F}_{q,\alpha}(u)$ the Fr\'echet differential of $\mathfrak{F}_{q,\alpha}$ at $u$. This is the linear continuous functional defined on $W^{1,2}_0(\Omega)$ by
\[
d\,\mathfrak{F}_{q,\alpha}(u)[\varphi]:=\int_\Omega \langle \nabla u,\nabla\varphi\rangle\,dx-\alpha\,\int_\Omega |u|^{q-2}\,u\,\varphi\,dx,\qquad \mbox{ for every } \varphi\in W^{1,2}_0(\Omega).
\]
\begin{lm}
\label{lm:PS}
Let $1<q<2$ and $\alpha>0$. Let $\Omega\subset \mathbb{R}^N$ be an open bounded set. The functional
$\mathfrak{F}_{q,\alpha}$ is of class $C^1$ and verifies the so-called {\rm Palais-Smale condition}, i.\,e. from every sequence $\{u_n\}_{n\in\mathbb{N}}\subset W^{1,2}_0(\Omega)$ satisfying the following properties:
\begin{enumerate}
\item $|\mathfrak{F}_{q,\alpha}(u_n)|\le C$, for every $n\in\mathbb{N}$;
\vskip.2cm
\item it holds
\[
\sup_{\|\varphi\|_{W^{1,2}_0(\Omega)}=1}\Big|d\,\mathfrak{F}_{q,\alpha}(u_n)[\varphi]\Big|=o(1),\qquad \mbox{ as } n\to\infty;
\]
\end{enumerate}
we can extract a subsequence strongly converging in $W^{1,2}_0(\Omega)$.
\end{lm}
\begin{proof}
The fact that $\mathfrak{F}_{q,\alpha}$ is of class $C^1$ is easily shown, see for example \cite[Theorem C.1]{St}. In order to verify the Palais-Smale condition, we take a sequence $\{u_n\}_{n\in\mathbb{N}}\subset W^{1,2}_0(\Omega)$ satisfying the properties above. We first observe that $\mathfrak{F}_{q,\alpha}$ is coercive on $W^{1,2}_0(\Omega)$, thanks to the fact that for every $\varepsilon>0$ we have
\[
\mathfrak{F}_{q,\alpha}(\varphi)\ge \frac{1}{2}\,\int_\Omega |\nabla \varphi|^2\,dx-\frac{\varepsilon}{2}\,\left(\int_\Omega |\varphi|^q\,dx\right)^\frac{2}{q}-\frac{2-q}{ 2\,q}\,\alpha^\frac{2}{2-q}\,\varepsilon^{-\frac{q}{2-q}},
\]
by Young's inequality with exponents $2/q$ and $2/(2-q)$. By choosing
\[
\varepsilon=\frac{\lambda_1(\Omega;q)}{2},
\]
and using the Sobolev-Poincar\'e inequality \eqref{sobolevpoincare}, we can infer
\begin{equation}
\label{coercive}
\mathfrak{F}_{q,\alpha}(\varphi)\ge \frac{1}{4}\,\int_\Omega |\nabla \varphi|^2\,dx-C,
\end{equation}
for some constant $C=C(N,q,\alpha,\Omega)>0$. This gives the claimed coercivity of our functional.
\par
Property (1), estimate \eqref{coercive} and the compactness of the embedding $W^{1,2}_0(\Omega)\hookrightarrow L^q(\Omega)$ imply that $u_n$ converges strongly in $L^q(\Omega)$ and weakly in $W^{1,2}_0(\Omega)$ to some $u\in W^{1,2}_0(\Omega)$, up to a subsequence. In order to prove that the convergence is actually strong in $W^{1,2}_0(\Omega)$, we write
\par
\[
\begin{split}
d\,\mathfrak{F}_{q,\alpha}(u_n)\left[u_n-u\right]-d\,\mathfrak{F}_{q,\alpha}(u)\left[u_n-u\right]&=\|u_n-u\|^2_{W^{1,2}_0(\Omega)}\\
&-\alpha\,\int_\Omega \Big(|u_n|^{q-2}\,u_n-|u|^{q-2}\,u\Big)\,(u_n-u)\,dx.
\end{split}
\]
This in turn implies that
\[
\begin{split}
\|u_n-u\|^2_{W^{1,2}_0(\Omega)}&\le \Big|d\,\mathfrak{F}_{q,\alpha}(u_n)\left[u_n-u\right]\Big|+\Big|d\,\mathfrak{F}_{q,\alpha}(u)\left[u_n-u\right]\Big|\\
&+\alpha\,\int_\Omega \Big(|u_n|^{q-1}+|u|^{q-1}\Big)\,|u_n-u|\,dx.
\end{split}
\]
We now observe that
\[
\lim_{n\to\infty}\Big|d\,\mathfrak{F}_{q,\alpha}(u_n)\left[u_n-u\right]\Big|=0,
\]
thanks to property (2), the uniform bound on the norm of $u_n$ given by \eqref{coercive} and the linearity of $d\mathfrak{F}_{q,\alpha}(u_n)$. Moreover, we also have
\[
\lim_{n\to\infty} \Big|d\,\mathfrak{F}_{q,\alpha}(u)\left[u_n-u\right]\Big|=0,
\]
by weak convergence of $u_n$ to $u$ in $W^{1,2}_0(\Omega)$. Finally, the strong $L^q$ convergence permits to infer that
\[
\lim_{n\to\infty}\int_\Omega \Big(|u_n|^{q-1}+|u|^{q-1}\Big)\,|u_n-u|\,dx=0,
\]
as well. This concludes the proof.
\end{proof}

\section{The ground state level}
\label{sec:3}

The set $\mathrm{Crit}(\mathfrak{F}_{q,\alpha})\subset (-\infty,0]$ is bounded, since our functional has a global minimum. This is the content of the next result.
\begin{prop}
\label{prop:min}
Let $1<q<2$ and $\alpha>0$. Let $\Omega\subset\mathbb{R}^N$ be an open bounded connected set. Then the functional $\mathfrak{F}_{q,\alpha}$ admits exactly two minimizers on $W^{1,2}_0(\Omega)$, given by $w$ and $-w$. Moreover, we have
\[
w\in L^\infty(\Omega)\qquad \mbox{ and }\qquad w>0\qquad \mbox{ on }\Omega.
\]
Finally, if $\Omega$ has $C^1$ boundary, then we also have $w\in C(\overline{\Omega})$.
\end{prop}
\begin{proof}
Existence of a minimizer is a standard fact, it is sufficient to use the Direct Method in the Calculus of Variations. Indeed, the functional is weakly lower semicontinuous and coercive on $W^{1,2}_0(\Omega)$ by \eqref{coercive}.
\par
We notice that any minimizer is not trivial. Indeed, for $\varphi\in W^{1,2}_0(\Omega)\setminus\{0\}$ and $t>0$, the value
\[
\mathfrak{F}_{q,\alpha}(t\,\varphi)=\frac{t^2}{2}\,\int_\Omega |\nabla \varphi|^2\,dx-\frac{\alpha}{q}\,t^q\,\int_\Omega |\varphi|^q\,dx,
\]
is strictly negative for $t$ sufficiently small, thanks to the fact that $q<2$. This shows that
\[
\min_{\varphi\in W^{1,2}_0(\Omega)}\left\{\frac{1}{2}\,\int_\Omega |\nabla \varphi|^2\,dx-\frac{\alpha}{q}\,\int_\Omega |\varphi|^q\,dx\right\}<0,
\]
and thus $\varphi\equiv 0$ can not be a minimizer.
\par
The fact that there exists at least a positive minimizer easily follows from the fact that
\[
\mathfrak{F}_{q,\alpha}(\varphi)=\mathfrak{F}_{q,\alpha}(|\varphi|),\qquad \mbox{ for every } \varphi\in W^{1,2}_0(\Omega).
\]
Moreover, positive minimizers are unique, see for example \cite[Theorem 3.1]{BraFra}. This shows the existence of a unique positive minimizer $w$. This must solve the relevant Euler-Lagrange equation, given by \eqref{LE}. In particular, $w$ is a nontrivial weakly superharmonic function and thus it is strictly positive by the minimum principle.
By observing that
\[
\mathfrak{F}_{q,\alpha}(-w)=\mathfrak{F}_{q,\alpha}(w),
\]
we obtain that $-w$ is the unique negative minimizer.
Finally, the claimed regularity for $w$ follows from the classical Elliptic Regularity Theory.
\par
We are only left to show that any minimizer must have constant sign in $\Omega$. We can use an argument based on the minimum principle, as in \cite[Proposition 2.3]{BraFra}.
Let us suppose that there exists a minimizer $u\in W^{1,2}_0(\Omega)$ such that
\[
u^+\not\equiv 0\qquad \mbox{ and }\qquad u^-\not\equiv 0.
\]
We have already observed that $|u|$ is still a minimizer. Thus both $u$ and $|u|$ are weak solutions of the Lane-Emden equation, i.e. we have
\[
\int_\Omega \langle \nabla u,\nabla \varphi\rangle\,dx=\alpha\,\int_\Omega |u|^{q-2}\,u\,\varphi\,dx,\qquad \mbox{ for every } \varphi\in W^{1,2}_0(\Omega),
\]
and
\[
\int_\Omega \langle \nabla |u|,\nabla \varphi\rangle\,dx=\alpha\,\int_\Omega |u|^{q-1}\,\varphi\,dx,\qquad \mbox{ for every } \varphi\in W^{1,2}_0(\Omega).
\]
By taking a convex combination of these two identities, we get
\[
\int_\Omega \left\langle \nabla \frac{|u|+u}{2},\nabla \varphi\right\rangle\,dx=\alpha\,\int_\Omega |u|^{q-2}\,\frac{|u|+u}{2}\,\varphi\,dx,\qquad \mbox{ for every } \varphi\in W^{1,2}_0(\Omega).
\]
We now observe that
\[
u^+=\frac{|u|+u}{2},
\]
thus the previous identity implies in particular that $u^+\in W^{1,2}_0(\Omega)$ satisfies
\[
\int_\Omega \langle \nabla u^+,\nabla \varphi\rangle\,dx\ge 0,\qquad \mbox{ for every } \varphi\in W^{1,2}_0(\Omega) \mbox{ such that }\varphi\ge 0.
\]
This means that $u^+$ is a non-negative weakly superharmonic function. By the minimum principle and the fact that $u^+\not\equiv 0$, we get that $u^+>0$ almost everywhere\footnote{Here we use that $\Omega$ is connected.} in $\Omega$. This contradicts the assumption $u^-\not\equiv 0$.
\end{proof}
\begin{oss}[Disconneted sets]
\label{oss:disconnected}
If $\Omega\subset\mathbb{R}^N$ is an open bounded set with $k$ connected components $\Omega_1,\dots,\Omega_k$, then by locality of the functional $\mathfrak{F}_{q,\alpha}$ the problem
\[
\min_{\varphi\in W^{1,2}_0(\Omega)} \mathfrak{F}_{q,\alpha}(\varphi),
\]
has exactly $2^k$ solutions. These are given by
\[
\left\{\sum_{i=1}^k \delta_i\,w_i\, :\, \delta_i\in \{-1,1\},\ w_i\in W^{1,2}_0(\Omega_i) \mbox{ positive minimizer of } \mathfrak{F}_{q,\alpha}\right\}.
\]
We further observe that the positive minimizers are still unique, even in this case.
\end{oss}
In what follows, we indicate by
\[
\Lambda_1:=\min_{\varphi\in W^{1,2}_0(\Omega)} \mathfrak{F}_{q,\alpha}(\varphi),
\]
the minimal level of our energy functional $\mathfrak{F}_{q,\alpha}$.
The next result asserts that the only nontrivial critical points having constant sign are $w$ and $-w$.
\begin{prop}
\label{prop:constantsign}
Let $1<q<2$ and $\alpha>0$. Let $\Omega\subset\mathbb{R}^N$ be an open bounded connected set. If $u\in W^{1,2}_0(\Omega)$ is a non-trivial critical point of $\mathfrak{F}_{q,\alpha}$ with
\[
\mathfrak{F}_{q,\alpha}(u)>\Lambda_1,
\]
then we have
\[
u^+\not\equiv 0\qquad \mbox{ and }\qquad u^-\not\equiv 0.
\]
\end{prop}
\begin{proof}
Let us suppose on the contrary that $u$ has constant sign. Without loss of generality, we can assume that $u^-\equiv 0$. By the minimum principle, we get that
\[
u=u^+>0,\qquad \mbox{ a.\,e. on }\Omega.
\]
For every $\varepsilon>0$ and $\psi\in C^\infty_0(\Omega)$, we insert in \eqref{LEweak} the test function
\[
\varphi=\frac{|\psi|^q}{(\varepsilon+u)^{q-1}}.
\]
We obtain
\[
\begin{split}
\alpha\,\int_\Omega \frac{u^{q-1}}{(\varepsilon+u)^{q-1}}\,|\psi|^q\,dx&=\int_\Omega \left\langle \nabla u,\nabla \frac{|\psi|^q}{(\varepsilon+u)^{q-1}} \right\rangle\,dx\\
&= \int_\Omega \left\langle \nabla (\varepsilon+u),\nabla \frac{|\psi|^q}{(\varepsilon+u)^{q-1}} \right\rangle\,dx\\
&\le \int_\Omega |\nabla (\varepsilon+u)|^{2-q}\,|\nabla \psi|^q\,dx\\
&\le \left(\int_\Omega |\nabla u|^2\,dx\right)^\frac{2-q}{2}\,\left(\int_\Omega |\nabla \psi|^2\,dx\right)^\frac{q}{2}.
\end{split}
\]
The first inequality follows from the {\it generalized Picone inequality} of \cite[Propositon 2.9]{BFK}.
\par
By taking the limit as $\varepsilon$ goes to $0$, using the Dominated Convergence Theorem and the fact that $u>0$ almost everywhere on $\Omega$, we get
\[
\alpha\,\int_\Omega |\psi|^q\,dx\le \left(\int_\Omega |\nabla u|^2\,dx\right)^\frac{2-q}{2}\,\left(\int_\Omega |\nabla \psi|^2\,dx\right)^\frac{q}{2}.
\]
By using Young's inequality with exponents $2/q$ and $2/(2-q)$, this gives
\[
\alpha\,\int_\Omega |\psi|^q\,dx\le \frac{2-q}{2}\,\int_\Omega |\nabla u|^2\,dx+\frac{q}{2}\int_\Omega |\nabla \psi|^2\,dx.
\]
By dividing both sides by $q$, this can be recast as
\[
\left(\frac{1}{2}-\frac{1}{q}\right)\,\int_\Omega |\nabla u|^2\,dx\le \mathfrak{F}_{q,\alpha}(\psi),\qquad \mbox{ for every } \psi\in C^\infty_0(\Omega).
\]
If we now use that $u$ is a critical point and recall \eqref{ecrit}, the previous estimate tells us that
\[
\mathfrak{F}_{q,\alpha}(u)\le \mathfrak{F}_{q,\alpha}(\psi),\qquad \mbox{ for every } \psi\in C^\infty_0(\Omega).
\]
By density of $C^\infty_0(\Omega)$ in $W^{1,2}_0(\Omega)$ and using the assumption on $u$, we finally get
\[
\Lambda_1<\mathfrak{F}_{q,\alpha}(u)=\inf_{\psi\in C^\infty_0(\Omega)}\mathfrak{F}_{q,\alpha}(\psi)=\Lambda_1,
\]
which gives the desired contradiction.
\end{proof}
We will repeatedly use the following property of minimizing sequences for $\mathfrak{F}_{q,\alpha}$.
\begin{lm}
\label{lm:minseq}
Let $1<q<2$ and $\alpha>0$. Let $\Omega\subset\mathbb{R}^N$ be an open bounded connected set. If $\{\varphi_n\}_{n\in\mathbb{N}}\subset W^{1,2}_0(\Omega)$ is such that
\[
\lim_{n\to\infty} \mathfrak{F}_{q,\alpha}(\varphi_n)=\Lambda_1,
\]
then we have
\[
\mbox{ either }\ \lim_{n\to\infty} \|\varphi_n-w\|_{W^{1,2}_0(\Omega)}=0\ \mbox{ or }\ \lim_{n\to\infty} \|\varphi_n-(-w)\|_{W^{1,2}_0(\Omega)}=0,
\]
up to a subsequence.
\end{lm}
\begin{proof}
By assumption, we have that there exists a constant $C$ such that
\[
\mathfrak{F}_{q,\alpha}(\varphi_n)\le C,\qquad \mbox{ for every } n\in\mathbb{N}.
\]
By using again \eqref{coercive},
we can then infer that $\{\varphi_n\}_{n\in\mathbb{N}}$ is bounded in $W^{1,2}_0(\Omega)$. Thus, by compactness of the embedding $W^{1,2}_0(\Omega)\hookrightarrow L^q(\Omega)$, there exists $\varphi\in W^{1,2}_0(\Omega)$ such that (up to a subsequence) $\{\varphi_n\}_{n\in\mathbb{N}}$ converges weakly in $W^{1,2}_0(\Omega)$ and strongly in $L^q(\Omega)$ to $\varphi$. In particular, by the lower semicontinuity of $\mathfrak{F}_{q,\alpha}$, $\varphi$ must be a minimizer. By Proposition \ref{prop:min}, we have that either $\varphi=w$ or $\varphi=-w$. Moreover, we have
\[
\begin{split}
\lim_{n\to\infty} \left[\frac{1}{2}\,\int_\Omega |\nabla \varphi_n|^2\,dx-\frac{1}{2}\,\int_\Omega |\nabla w|^2\,dx\right]&=\lim_{n\to\infty} \Big[\mathfrak{F}_{q,\alpha}(\varphi_n)-\Lambda_1\Big]\\
&+\lim_{n\to\infty}\left[\frac{\alpha}{q}\,\int_\Omega |\varphi_n|^q\,dx-\frac{\alpha}{q}\,\int_\Omega |w|^q\,dx\right]=0.
\end{split}
\]
By uniform convexity of the $L^2$ norm, this permits to upgrade the weak convergence of $\{\varphi_n\}_{n\in\mathbb{N}}$ to the strong one.
\end{proof}
The following result will play a crucial role for the proof of our main result. It asserts that the minimal level $\Lambda_1$ is isolated in the ``spectrum'' $\mathrm{Crit}(\mathfrak{F}_{q,\alpha})$. For this, some regularity assumptions on $\partial\Omega$ are needed.
\begin{prop}[Fundamental gap]
\label{prop:dasolo}
Let $1<q<2$ and $\alpha>0$. Let $\Omega\subset\mathbb{R}^N$ be an open bounded connected set, with $C^1$ boundary. By setting
\[
\Lambda_2=\inf\Big\{\lambda\in\mathrm{Crit}(\mathfrak{F}_{q,\alpha})\, :\, \lambda>\Lambda_1\Big\},
\]
we have that:
\begin{enumerate}
\item $\Lambda_2>\Lambda_1$;
\vskip.2cm
\item $\Lambda_2\in \mathrm{Crit}(\mathfrak{F}_{q,\alpha})$.
\end{enumerate}
\end{prop}
\begin{proof}
We first observe that
\[
\Big\{\lambda\in\mathrm{Crit}(\mathfrak{F}_{q,\alpha})\, :\, \lambda>\Lambda_1\Big\}\not=\emptyset,
\]
since $\Lambda_1<0$ and $0$ is a critical value of $\mathfrak{F}_{q,\alpha}$, corresponding to the trivial critical point $u\equiv 0$.
\par
We argue by contradiction and assume that $\Lambda_2=\Lambda_1$. This implies that there is a sequence $\{\lambda_n\}_{n\in\mathbb{N}}$ of critical values, such that
\[
\lambda_n>\Lambda_1\ \mbox{ for every }n\in\mathbb{N}\qquad \mbox{ and }\qquad \lim_{n\to\infty} \lambda_n=\Lambda_1.
\]
Let $\{u_n\}_{n\in\mathbb{N}}\subset W^{1,2}_0(\Omega)$ be a corresponding sequence of critical points, i.e. weak solutions of \eqref{LE}. We thus have
\[
\Lambda_1=\lim_{n\to\infty}\lambda_n=\mathfrak{F}_{q,\alpha}(u_n).
\]
By Lemma \ref{lm:minseq}, we get that $\{u_n\}_{n\in\mathbb{N}}$ converges strongly in $W^{1,2}_0(\Omega)$ either to $w$ or to $-w$. However, under the standing assumptions on $\Omega$, we know that $\{w,\,-w\}$ are isolated critical points for $\mathfrak{F}_{q,\alpha}$ with respect to the strong $L^1(\Omega)$ topology, thanks to \cite[Theorem A]{BDF}. Thus we reach the desired contradiction.
\vskip.2cm\noindent
We now show that $\Lambda_2$ is a critical value.
We know by Lemma \ref{lm:closed} that $\mathrm{Crit}(\mathfrak{F}_{q,\alpha})$ is a closed subset of $\mathbb{R}$. This is enough to conclude that the infimum in the definition of $\Lambda_2$ must actually be a minimum.
\end{proof}
\begin{oss}
\label{oss:BDF}
It is precisely in the previous result that the smoothness assumption on $\partial\Omega$ enters. This is needed to apply \cite[Theorem A]{BDF}. However, as pointed out in \cite{BDF}, such a result would also allow for Lipschitz sets: in this case, it may happen that the result is only valid in a restricted range
\[
q_\Omega<q<2,
\]
with $q_\Omega\ge 1$ depending on the Lipschitz constant of $\partial\Omega$. We refer the reader to \cite[Remark 1.2]{BDF} for a thorough discussion.
\par
We also point out that the previous result still holds if $\Omega$ has a finite number connected components, while for a set with countably infinite connected components it holds $\Lambda_2=\Lambda_1$.
\end{oss}
We conclude this section by showing that above the minimal value of $\mathfrak{F}_{q,\alpha}$, we can always find sign-changing functions. This is quite straightforward, we include it to stress that the statement of Theorem \ref{teo:main} permits indeed to take sign-changing initial data.
\begin{lm}
\label{lm:nontrivial}
Let $1<q<2$ and $\alpha>0$. Let $\Omega\subset\mathbb{R}^N$ be an open bounded set with a finite number of connected components,  such that
\[
\Lambda_2>\Lambda_1.
\]
Then there exists infinitely many sign-changing functions $\varphi\in W^{1,2}_0(\Omega)$ such that
\[
\Lambda_1<\mathfrak{F}_{q,\alpha}(\varphi)< \Lambda_2,
\]
and
\[
\mathfrak{F}_{q,\alpha}\left(\varphi^-\right)\ge 0.
\]
\end{lm}
\begin{proof}
For every $\varepsilon>0$ small enough, we consider the sets $\Omega_\varepsilon=\{x\in\Omega\, :\, \mathrm{dist}(x,\partial\Omega)>\varepsilon\}$. We then claim that
\[
\Lambda_1(\varepsilon)=\min_{\varphi\in W^{1,2}_0(\Omega_\varepsilon)} \mathfrak{F}_{q,\alpha}(\varphi),
\]
is such that
\begin{equation}
\label{convergelambda}
\lim_{\varepsilon\to 0^+} \Lambda_1(\varepsilon)=\Lambda_1.
\end{equation}
Indeed, since $\Omega_\varepsilon\subset\Omega$, we can view $W^{1,2}_0(\Omega_\varepsilon)$ as a subspace of $W^{1,2}_0(\Omega)$, by extending its elements by $0$ outside. Thus we get
\[
\Lambda_1(\varepsilon)\ge \Lambda_1\quad \mbox{ which implies }\quad \liminf_{\varepsilon\to 0^+} \Lambda_1(\varepsilon)\ge \Lambda_1.
\]
On the other hand, by density of $C^\infty_0(\Omega)$ in $W^{1,2}_0(\Omega)$, for every $\delta>0$ there exists $\varphi_\delta\in C^\infty_0(\Omega)$ such that
\[
\mathfrak{F}_{q,\alpha}(\varphi_\delta)<\Lambda_1+\delta.
\]
Since $\varphi_\delta$ has compact support in $\Omega$, there exists $\varepsilon_0$ such that $\varphi_\delta\in C^\infty_0(\Omega_\varepsilon)$ for every $0<\varepsilon\le \varepsilon_0$.  This implies that
\[
\Lambda_1(\varepsilon)\le \mathfrak{F}_{q,\alpha}(\varphi_\delta)<\Lambda_1+\delta,\qquad \mbox{ for every } 0<\varepsilon\le \varepsilon_0,
\]
and thus
\[
\limsup_{\varepsilon\to 0^+} \Lambda_1(\varepsilon)\le \Lambda_1+\delta.
\]
By arbitrariness of $\delta>0$, we finally obtain \eqref{convergelambda}, as claimed.\par
In light of \eqref{convergelambda}, the assumption $\Lambda_2>\Lambda_1$ entails that we can choose $\widetilde\varepsilon>0$ such that
\[
\Lambda_1(\widetilde\varepsilon)<\Lambda_2.
\]
Accordingly, by taking $w_{\widetilde{\varepsilon}}$ the positive minimizer of $\mathfrak{F}_{q,\alpha}$ over $W^{1,2}_0(\Omega_{\widetilde\varepsilon})$, we have
\[
\mathfrak{F}_{q,\alpha}(w_{\widetilde{\varepsilon}})<\Lambda_2.
\]
We observe that, by construction, there exists a ball $B_R(x_0)\subset\Omega$ such that
\[
w_{\widetilde{\varepsilon}}=0\qquad \mbox{ a.\,e. on }B_R(x_0).
\]
We now pick a nonnegative function $\psi\in C^\infty_0(B_1(0))\setminus\{0\}$ and define for $0<r<R$
\[
\psi_r(x)=r\,\psi\left(\frac{x-x_0}{r}\right).
\]
Its energy is given by
\begin{equation}
\label{energiapn}
\mathfrak{F}_{q,\alpha}(\psi_r)=\frac{r^N}{2}\,\int_{B_1(0)} |\nabla \psi|^2\,dx-\frac{\alpha\,r^{q+N}}{q}\,\int_{B_1(0)} |\psi|^q\,dx,
\end{equation}
so that in particular
\[
\lim_{r\to 0^+} \mathfrak{F}_{q,\alpha}(\psi_r)=0.
\]
We can then choose $0<r_0<R$ such that
\[
\mathfrak{F}_{q,\alpha}(\psi_{r_0})<\Lambda_2-\mathfrak{F}_{q,\alpha}(w_{\widetilde{\varepsilon}}),
\]
the latter being a positive quantity, as already said. If we now define
\[
\varphi=w_{\widetilde{\varepsilon}}-\psi_{r_0},
\]
use that the two functions have disjoint supports and the locality of the functional $\mathfrak{F}_{q,\alpha}$, we get the desired conclusion
\[
\Lambda_1<\mathfrak{F}_{q,\alpha}(\varphi)< \Lambda_2.
\]
Observe that by construction $\varphi^-=\psi_{r_0}$, thus by \eqref{energiapn} and by using that $r^{q+N}=o(r^N)$ as $r$ goes to $0$, we can always suppose that
\[
\mathfrak{F}_{q,\alpha}(\varphi^-)=\mathfrak{F}_{q,\alpha}(\psi_{r_0})\ge 0,
\]
up to further refine the choice of $r_0<R$. This concludes the proof.
\end{proof}

\section{The mountain pass level}
\label{sec:4}
We start this section with a simple result, useful to show that the functional $\mathfrak{F}_{q,\alpha}$ enjoys a \emph{mountain pass} structure.
\begin{lm}
\label{lm:valli}
Let $1<q<2$ and $\alpha>0$. Let $\Omega\subset\mathbb{R}^N$ be an open bounded connected set. We define
\[
0<\ell:=\|w-(-w)\|_{W^{1,2}_0(\Omega)}=2\,\|w\|_{W^{1,2}_0(\Omega)},
\]
then there exists a constant $C>\Lambda_1$
such that
\[
\mathfrak{F}_{q,\alpha}(\varphi)\ge C,\qquad \mbox{ for every } \varphi\in W^{1,2}_0(\Omega) \mbox{ such that } \|\varphi-w\|_{W^{1,2}_0(\Omega)}=\frac{\ell}{2}.
\]
\end{lm}
\begin{proof}
We argue by contradiction and assume that there exists a sequence $\{\varphi_n\}_{n\in\mathbb{N}}\subset W^{1,2}_0(\Omega)$ such that
\begin{equation}
\label{amezzo}
\|\varphi_n-w\|_{W^{1,2}_0(\Omega)}=\frac{\ell}{2},\qquad \mbox{ for every } n\in\mathbb{N}
\end{equation}
and
\[
\lim_{n\to\infty}\mathfrak{F}_{q,\alpha}(\varphi_n)=\Lambda_1.
\]
By Lemma \ref{lm:minseq}, we know that
\[
\mbox{ either } \lim_{n\to\infty} \|\varphi_n-w\|_{W^{1,2}_0(\Omega)}=0\quad \mbox{ or }\quad \lim_{n\to\infty} \|\varphi_n-(-w)\|_{W^{1,2}_0(\Omega)}=0,
\]
up to a subsequence. Both possibilities contradict \eqref{amezzo} and thus the claim follows.
\end{proof}
We can now prove the main result of this section.
\begin{teo}[The mountain pass level]
\label{teo:MP}
Let $1<q<2$ and $\alpha>0$. Let $\Omega\subset\mathbb{R}^N$ be an open bounded connected set. If we set
\[
\Gamma=\Big\{\gamma\in C([0,1];W^{1,2}_0(\Omega))\, :\, \gamma(0)=w,\, \gamma(1)=-w\Big\},
\]
then the value
\begin{equation}
\label{PM}
\Lambda^*:=\inf_{\gamma\in \Gamma}\max_{\varphi\in \mathrm{Im}(\gamma)} \mathfrak{F}_{q,\alpha}(\varphi)
\end{equation}
is a critical value of $\mathfrak{F}_{q,\alpha}$. Moreover, if we set\footnote{By the symbol $\gamma(+\infty)=-w$ we intend that
\[
\lim_{t\to+\infty}\|\gamma(t)-(-w)\|_{L^q(\Omega)}=0,
\]
while $\mathrm{Im}(\gamma)$ denotes the image of $\gamma$.}
\[
\widetilde\Gamma=\Big\{\gamma\in C([0,1];L^q(\Omega))\, :\, \mathrm{Im}(\gamma)\subset W^{1,2}_0(\Omega),\,\gamma(0)=w,\, \gamma(1)=-w\Big\},
\]
\[
\widetilde\Gamma_\infty=\Big\{\gamma\in C([0,+\infty);L^q(\Omega))\, :\, \mathrm{Im}(\gamma)\subset W^{1,2}_0(\Omega),\,\gamma(0)=w,\, \gamma(+\infty)=-w\Big\},
\]
and
\[
\mu^*:=\inf_{\gamma\in \widetilde\Gamma}\sup_{\varphi\in \mathrm{Im}(\gamma)} \mathfrak{F}_{q,\alpha}(\varphi)\qquad \mbox{ and }\qquad \mu^*_\infty:=\inf_{\gamma\in \widetilde\Gamma_\infty}\sup_{\varphi\in \mathrm{Im}(\gamma)} \mathfrak{F}_{q,\alpha}(\varphi),
\]
then
\[
\mu^*=\mu^*_\infty=\Lambda^*.
\]
\end{teo}
\begin{proof}
We divide the proof in three parts: we first prove that $\Lambda^*$ defined by \eqref{PM} is a critical value for $\mathfrak{F}_{q,\alpha}$. Then we prove separately that
\[
\mu^*=\Lambda^*\qquad \mbox{ and }\qquad \mu^*_\infty=\mu^*.
\]
{\bf Part 1: $\Lambda^*$ is critical}.
We have already observed that $\mathfrak{F}_{q,\alpha}$ is a $C^1$ functional which satisfies the Palais-Smale condition, see Lemma \ref{lm:PS}. Moreover, Lemma \ref{lm:valli} guarantees that $\mathfrak{F}_{q,\alpha}$ has a {\it mountain pass structure}, that is we have:
\begin{itemize}
\item $\mathfrak{F}_{q,\alpha}(w)=\Lambda_1$;
\vskip.2cm
\item $\mathfrak{F}_{q,\alpha}(\varphi)\ge C>\Lambda_1$, for every $\varphi\in W^{1,2}_0(\Omega)$ such that $\|\varphi-w\|_{W^{1,2}_0(\Omega)}=\ell/2$;
\vskip.2cm
\item for $-w$ we have
\[
\|(-w)-w\|_{W^{1,2}_0(\Omega)}>\frac{\ell}{2}\qquad \mbox{ and }\qquad \mathfrak{F}_{q,\alpha}(-w)=\Lambda_1<C.
\]
\end{itemize}
It is then sufficient to apply \cite[Chapter II, Theorem 6.1]{St}, we leave the details to the reader.
\vskip.2cm\noindent
{\bf Part 2:} $\mu^*=\Lambda^*$.
It is easily seen that $\Gamma\subset \widetilde\Gamma$, thus it is immediate to get
\[
\mu^*\le \Lambda^*.
\]
In order to prove the reverse inequality, for every $\varepsilon>0$ we take $\gamma_\varepsilon\in \widetilde\Gamma$ such that
\begin{equation}
\label{energiapath}
\mu^*+\varepsilon>\sup_{\varphi\in \mathrm{Im}(\gamma_\varepsilon)} \mathfrak{F}_{q,\alpha}(\varphi).
\end{equation}
We observe that $\gamma_\varepsilon$ is uniformly continuous on $[0,1]$.
Thus, if we fix $\delta>0$, by uniform continuity there exists $\eta>0$ such that if $|t-s|<\eta$, we have
\[
\|\gamma_\varepsilon(t)-\gamma_\varepsilon(s)\|_{L^q(\Omega)}<\delta.
\]
We take a partition $\{t_0,\dots,t_k\}$ of $[0,1]$ such that
\[
t_0=0,\qquad t_k=1,\qquad |t_i-t_{i+1}|<\eta, \mbox{ for every } i=0,\dots,k-1,
\]
then we define the new curve $\theta_\varepsilon:[0,1]\to W^{1,2}_0(\Omega)$, which is given by the piecewise affine interpolation of the points $\gamma_\varepsilon(t_0),\gamma_\varepsilon(t_1),\dots,\gamma_\varepsilon(t_k)$. More precisely, we have
\[
\theta_\varepsilon(t)=\left(1-\frac{t-t_i}{t_{i+1}-t_i}\right)\,\gamma_\varepsilon(t_i)+\frac{t-t_i}{t_{i+1}-t_i}\,\gamma_\varepsilon(t_{i+1}),\qquad \mbox{ for every } t\in[t_i,t_{i+1}].
\]
Observe that by construction, we have
\[
\theta_\varepsilon\in C([0,1];W^{1,2}_0(\Omega)),
\]
thus this curve is admissible for the variational formulation of the mountain pass level $\Lambda^*$. In order to conclude, we need to estimate
\[
\max_{\varphi\in\mathrm{Im}(\theta_\varepsilon)} \mathfrak{F}_{q,\alpha}(\varphi).
\]
We estimate the energy of the path on each interval $[t_i,t_{i+1}]$: for simplicity, we set
\[
\tau=\frac{t-t_i}{t_{i+1}-t_i}.
\]
By using the definition of $\mathfrak{F}_{q,\alpha}$ and the convexity of the Dirichlet integral, we get for every $t\in[t_i,t_{i+1}]$
\[
\begin{split}
\mathfrak{F}_{q,\alpha}(\theta_\varepsilon(t))&=\frac{1}{2}\,\int_\Omega |\nabla ((1-\tau)\,\gamma_\varepsilon(t_i)+\tau\,\gamma_\varepsilon(t_{i+1}))|^2\,dx -\frac{\alpha}{q}\,\int_\Omega |(1-\tau)\,\gamma_\varepsilon(t_i)+\tau\,\gamma_\varepsilon(t_{i+1})|^q\,dx\\
&\le (1-\tau)\,\frac{1}{2}\,\int_\Omega |\nabla \gamma_\varepsilon(t_i)|^2\,dx+\tau\,\frac{1}{2}\,\int_\Omega |\nabla \gamma_\varepsilon(t_{i+1})|^2\,dx\\
&-\frac{\alpha}{q}\,\int_\Omega |(1-\tau)\,\gamma_\varepsilon(t_i)+\tau\,\gamma_\varepsilon(t_{i+1})|^q\,dx\\
&=(1-\tau)\,\mathfrak{F}_{q,\alpha}(\gamma_\varepsilon(t_i))+\tau\,\mathfrak{F}_{q,\alpha}(\gamma_\varepsilon(t_{i+1}))\\
&+\frac{\alpha}{q}\,\left[(1-\tau)\,\int_\Omega |\gamma_\varepsilon(t_i)|^q\,dx+\tau\,\int_\Omega |\gamma_\varepsilon(t_{i+1})|^q\,dx\right]\\
&-\frac{\alpha}{q}\,\int_\Omega |(1-\tau)\,\gamma_\varepsilon(t_i)+\tau\,\gamma_\varepsilon(t_{i+1})|^q\,dx.
\end{split}
\]
We set for brevity
\[
\mathcal{R}_1=\left[(1-\tau)\,\int_\Omega |\gamma_\varepsilon(t_i)|^q\,dx+\tau\,\int_\Omega |\gamma_\varepsilon(t_{i+1})|^q\,dx\right]
\]
and
\[
\mathcal{R}_2=\int_\Omega |(1-\tau)\,\gamma_\varepsilon(t_i)+\tau\,\gamma_\varepsilon(t_{i+1})|^q\,dx,
\]
then by using \eqref{energiapath}, we obtain
\begin{equation}
\label{stimaR}
\mathfrak{F}_{q,\alpha}(\theta_\varepsilon(t))< \mu^*+\varepsilon+\frac{\alpha}{q}\,\Big(\mathcal{R}_1-\mathcal{R}_2\Big),\qquad \mbox{ for every } t\in[t_i,t_{i+1}].
\end{equation}
We now need to estimate the remainder terms $\mathcal{R}_1$ and $\mathcal{R}_2$. We rewrite them as follows
\[
\mathcal{R}_1=\int_\Omega |\gamma_\varepsilon(t_i)|^q\,dx+\tau\,\left(\int_\Omega |\gamma_\varepsilon(t_{i+1})|^q\,dx-\int_\Omega |\gamma_\varepsilon(t_i)|^q\,dx\right),
\]
and
\[
\mathcal{R}_2=\int_\Omega |\gamma_\varepsilon(t_i)+\tau\,(\gamma_\varepsilon(t_{i+1})-\gamma_\varepsilon(t_i))|^q\,dx.
\]
In order to estimate $\mathcal{R}_2$, we observe that by convexity of the map $\tau\mapsto |\tau|^q$, we have
\[
|\gamma_\varepsilon(t_i)+\tau\,(\gamma_\varepsilon(t_{i+1})-\gamma_\varepsilon(t_i))|^q\ge |\gamma_\varepsilon(t_i)|^q+q\,\tau\,|\gamma_\varepsilon(t_i)|^{q-2}\,\gamma_\varepsilon(t_i)\,(\gamma_\varepsilon(t_{i+1})-\gamma_\varepsilon(t_i)).
\]
By integrating this over $\Omega$, we get, recalling the choice of the width of the partition,
\[
\begin{split}
-\mathcal{R}_2&\le -\int_\Omega |\gamma_\varepsilon(t_i)|^q\,dx-q\,\tau\,\int_\Omega |\gamma_\varepsilon(t_i)|^{q-2}\,\gamma_\varepsilon(t_i)\,(\gamma_\varepsilon(t_{i+1})-\gamma_\varepsilon(t_i))\,dx\\
&\le-\int_\Omega |\gamma_\varepsilon(t_i)|^q\,dx+q\,\tau\,\int_\Omega |\gamma_\varepsilon(t_i)|^{q-1}\,|\gamma_\varepsilon(t_{i+1})-\gamma_\varepsilon(t_i)|\,dx\\
&\le -\int_\Omega |\gamma_\varepsilon(t_i)|^q\,dx+q\,\tau\,\left(\int_\Omega |\gamma_\varepsilon(t_i)|^q\,dx\right)^\frac{q-1}{q}\,\|\gamma_\varepsilon(t_{i+1})-\gamma_\varepsilon(t_i)\|_{L^q(\Omega)}\\
&\le-\int_\Omega |\gamma_\varepsilon(t_i)|^q\,dx+q\,\tau\,\left(\int_\Omega |\gamma_\varepsilon(t_i)|^q\,dx\right)^\frac{q-1}{q}\,\delta.
\end{split}
\]
We observe that, by using the Sobolev-Poincar\'e inequality \eqref{sobolevpoincare}, the coercivity estimate \eqref{coercive} and the assumption \eqref{energiapath}, we have
\begin{equation}
\label{norms}
\left(\int_\Omega |\gamma_\varepsilon(t_i)|^q\,dx\right)^\frac{q-1}{q}\le C,
\end{equation}
for some uniform constant $C>0$. Thus in conclusion we get
\[
-\mathcal{R}_2\le -\int_\Omega |\gamma_\varepsilon(t_i)|^q\,dx+C\,\delta,
\]
In order to estimate $\mathcal{R}_1$, we use the elementary inequality
\[
|a^q-b^q|\le q\,(a^{q-1}+b^{q-1})\,|a-b|,\qquad \mbox{ for every }a,b\ge 0,
\]
see inequality \eqref{egolo0}  below.
Then by using this with
\[
a=\|\gamma_\varepsilon(t_i)\|_{L^q(\Omega)}\qquad \mbox{ and }\qquad b=\|\gamma_\varepsilon(t_{i+1})\|_{L^q(\Omega)},
\]
we obtain
\[
\begin{split}
\mathcal{R}_1&\le \int_\Omega |\gamma_\varepsilon(t_i)|^q\,dx+\left|\int_\Omega |\gamma_\varepsilon(t_{i+1})|^q\,dx-\int_\Omega |\gamma_\varepsilon(t_i)|^q\,dx\right|\\
&\le \int_\Omega |\gamma_\varepsilon(t_i)|^q\,dx+ C\,\left(\|\gamma_\varepsilon(t_i)\|_{L^q(\Omega)}^{q-1}+\|\gamma_\varepsilon(t_{i+1})\|_{L^q(\Omega)}^{q-1}\right)\,\left|\|\gamma_\varepsilon(t_i)\|_{L^q(\Omega)}-\|\gamma_\varepsilon(t_{i+1})\|_{L^q(\Omega)}\right|\\
&\le \int_\Omega |\gamma_\varepsilon(t_i)|^q\,dx+C\,\left(\|\gamma_\varepsilon(t_i)\|_{L^q(\Omega)}^{q-1}+\|\gamma_\varepsilon(t_{i+1})\|_{L^q(\Omega)}^{q-1}\right)\,\|\gamma_\varepsilon(t_i)-\gamma_\varepsilon(t_{i+1})\|_{L^q(\Omega)}.
\end{split}
\]
By using again \eqref{norms} and the choice of the width of the partition, we obtain
\[
\mathcal{R}_1\le \int_\Omega |\gamma_\varepsilon(t_i)|^q\,dx+C\,\delta.
\]
In conclusion, we get
\[
\mathcal{R}_1-\mathcal{R}_2\le C\,\delta,
\]
for some constant $C>0$ independent of $\delta$. By using this in \eqref{stimaR}, we get
\[
\mathfrak{F}_{q,\alpha}(\theta_\varepsilon(t))< \mu^*+\varepsilon+C\,\delta,\qquad \mbox{ for every } t\in[t_i,t_{i+1}].
\]
Such an estimate finally holds for every $t\in [0,1]$, by arbitrariness of $i\in\{1,\dots,k-1\}$.
Therefore
\[
\Lambda^*\le \max_{\varphi\in\mathrm{Im}(\theta_\varepsilon)} \mathfrak{F}_{q,\alpha}(\varphi)< \mu^*+\varepsilon+C\,\delta.
\]
and by taking the limit as $\delta$ goes to $0$, we obtain
\[
\Lambda^*\le \mu^*+\varepsilon.
\]
By the arbitrariness of $\varepsilon>0$, we finally get that $\Lambda^*\le \mu^*$, as well.
\vskip.2cm\noindent
{\bf Part 3:} $\mu^*_\infty=\mu^*$.
We can identify each path $\gamma\in \widetilde\Gamma$ with the curve $\widetilde\gamma$ of $\widetilde\Gamma_\infty$ given by
\[
\widetilde\gamma(t)=\left\{\begin{array}{cc}
\gamma(t),& \mbox{ if } t\in[0,1],\\
-w,&\mbox{ if } t>1.
\end{array}
\right.
\]
Of course, we have
\[
\sup_{\varphi\in \mathrm{Im}(\gamma)} \mathfrak{F}_{q,\alpha}(\varphi)=\sup_{\varphi\in \mathrm{Im}(\widetilde\gamma)} \mathfrak{F}_{q,\alpha}(\varphi).
\]
Thus, through this identification, we can say that $\widetilde\Gamma\subset \widetilde\Gamma_\infty$ and the inequality
\[
\mu^*\ge \mu^*_\infty,
\]
follows. In order to prove the reverse inequality, for every $\varepsilon>0$ we take $\gamma_\varepsilon\in \widetilde\Gamma_\infty$ such that
\[
\mu^*_\infty+\varepsilon> \sup_{\varphi\in \mathrm{Im}(\gamma_\varepsilon)} \mathfrak{F}_{q,\alpha}(\varphi).
\]
We fix $\delta>0$, by assumption there exists $M>0$ such that
\[
\|\gamma_\varepsilon(M)-(-w)\|_{L^q(\Omega)}<\delta.
\]
We then build an element of $\widetilde\Gamma$ as follows: at first, we take
\[
\gamma^M_\varepsilon(t)=\left\{\begin{array}{ll}
\gamma_\varepsilon(t), & \mbox{ if } t\in [0,M],\\
(1-t+M)\,\gamma_\varepsilon(M)-(t-M)\,w,& \mbox{ if } t\in[M,M+1],
\end{array}
\right.
\]
then we rescale it, i.e. we define
\[
\theta_\varepsilon(t)=\gamma^M_\varepsilon\left((M+1)\,t\right),\qquad \mbox{ for every } t\in[0,1].
\]
By construction, we have $\theta_\varepsilon\in \widetilde\Gamma$. We have to estimate the energy on this path. We observe that for $t\in[0,M/(M+1)]$, we have
\[
\mathfrak{F}_{q,\alpha}(\theta_\varepsilon(t))\le \sup_{\varphi\in \mathrm{Im}(\gamma_\varepsilon)} \mathfrak{F}_{q,\alpha}(\varphi)<\mu^*_\infty+\varepsilon.
\]
For $t\in [M/(M+1),1]$, the curve $\theta_\varepsilon$ is just the linear interpolation between $\gamma_\varepsilon(M)$ and the endpoint $-w$. By proceeding as in {\bf Part 2}, one can easily get
\[
\mathfrak{F}_{q,\alpha}(\theta_\varepsilon(t))\le \mu^*_\infty+\varepsilon+C\,\delta,
\]
thanks to the choice of $M$. The previous estimates entail that
\[
\mu^*\le \sup_{\varphi\in\mathrm{Im}(\theta_\varepsilon)} \mathfrak{F}_{q,\alpha}(\varphi)\le \mu^*_\infty+\varepsilon+C\,\delta.
\]
By taking the limit as $\delta$ goes to $0$, we obtain
\[
\mu^*\le \mu^*_\infty+\varepsilon.
\]
Finally, by the arbitrariness of $\varepsilon>0$, we get the desired conclusion.
\end{proof}
The continuity properties of the paths entering in the definition of $\Lambda^*$ can be further relaxed. Indeed, continuity in the $L^1$ strong topology is still sufficient. This is the content of the following
\begin{coro}
\label{coro:weakest}
Let $1<q<2$ and $\alpha>0$. Let $\Omega\subset\mathbb{R}^N$ be an open bounded connected set. If we set
\[
\widetilde\Gamma^1_\infty=\Big\{\gamma\in C([0,+\infty);L^1(\Omega))\, :\, \mathrm{Im}(\gamma)\subset W^{1,2}_0(\Omega),\,\gamma(0)=w,\, \gamma(+\infty)=-w\Big\},
\]
and
\[
 \nu^*_\infty:=\inf_{\gamma\in \widetilde\Gamma^1_\infty}\sup_{\varphi\in \mathrm{Im}(\gamma)} \mathfrak{F}_{q,\alpha}(\varphi),
\]
then
\[
\nu^*_\infty=\Lambda^*.
\]
\end{coro}
\begin{proof}
By using the notation of Theorem \ref{teo:MP}, it is sufficient to prove that
\[
\nu^*_\infty=\mu^*_\infty.
\]
We first notice that
\[
\widetilde\Gamma_\infty\subset\widetilde\Gamma_\infty^1,
\]
since by H\"older's inequality we have
\[
\|\gamma(t)-\gamma(s)\|_{L^1(\Omega)}\le |\Omega|^{1-\frac{1}{q}}\,\|\gamma(t)-\gamma(s)\|_{L^q(\Omega)},\qquad \mbox{ for }t,s\in[0,+\infty),\ \gamma\in \widetilde\Gamma_\infty.
\]
This implies that
\[
\nu^*_\infty\le \mu^*_\infty.
\]
In order to prove the reverse inequality, for every $\varepsilon>0$ we take $\gamma_\varepsilon\in \widetilde\Gamma_\infty^1$ such that
\[
\nu^*_\infty+\varepsilon>\sup_{\varphi\in\mathrm{Im}(\gamma_\varepsilon)} \mathfrak{F}_{q,\alpha}(\varphi).
\]
By using again the coercivity estimate \eqref{coercive}, we thus get
\[
\int_\Omega |\nabla \varphi|^2\,dx\le C,\qquad \mbox{ for every }\varphi\in \mathrm{Im}(\gamma_\varepsilon).
\]
We can combine this estimate with the interpolation inequality \eqref{gagliardo}, in order to get
\[
\begin{split}
\|\gamma_\varepsilon(t)-\gamma_\varepsilon(s)\|_{L^q(\Omega)}&\le \Big(\lambda_1(\Omega)\Big)^\frac{\vartheta-1}{2}\,\|\gamma_\varepsilon(t)-\gamma_\varepsilon(s)\|_{L^1(\Omega)}^\vartheta\,\|\nabla \gamma_\varepsilon(t)-\nabla \gamma_\varepsilon(s)\|_{L^2(\Omega)}^{1-\vartheta}\\
&\le \widetilde{C}\,\|\gamma_\varepsilon(t)-\gamma_\varepsilon(s)\|_{L^1(\Omega)}^\vartheta,\qquad \mbox{ for every }t,s\in [0,+\infty).
\end{split}
\]
Here $\widetilde{C}>0$ is a uniform constant. This shows that $\gamma_\varepsilon\in \widetilde\Gamma_\infty$, as well. Thus, we obtain
\[
\nu^*_\infty+\varepsilon>\sup_{\varphi\in\mathrm{Im}(\gamma_\varepsilon)} \mathfrak{F}_{q,\alpha}(\varphi)\ge \mu^*_\infty.
\]
Since $\varepsilon>0$ was arbitrary, we get the desired conclusion.
\end{proof}
In the next result, we exclude that the mountain pass level $\Lambda^*$ collapses to the ground state level $\Lambda_1$, under the standing assumptions.
\begin{prop}
\label{prop:nocollapse}
Under the assumptions of Theorem \ref{teo:MP}, we have $\Lambda^*>\Lambda_1$.
In particular, we get
\[
\Lambda^*\ge \Lambda_2.
\]
\end{prop}
\begin{proof}
We argue by contradiction and assume that $\Lambda^*=\Lambda_1$. Then by definition, for every $n\in\mathbb{N}$ there exists a curve $\gamma_n\in \Gamma$ such that
\[
\max_{\varphi\in \mathrm{Im}(\gamma_n)}\mathfrak{F}_{q,\alpha}(\varphi)<\Lambda_1+\frac{1}{n+1}.
\]
The continuity of $\gamma$ entails that $\mathrm{Im}(\gamma_n)$ is a connected set, while the two neighborhoods
\[
B_w:=\Big\{\varphi\in W^{1,2}_0(\Omega)\, :\, \|\varphi-w\|_{W^{1,2}_0(\Omega)}\le \frac{\ell}{4}\Big\},
\]
and
\[
B_{-w}:=\Big\{\varphi\in W^{1,2}_0(\Omega)\, :\, \|\varphi-(-w)\|_{W^{1,2}_0(\Omega)}\le \frac{\ell}{4}\Big\},
\]
are disjoint. Here $\ell$ still denotes the distance between $w$ and $-w$.
Then it is possible to choose $\varphi_n\in \mathrm{Im}(\gamma_n)$ such that
\[
\varphi_n\not\in B_w\cup B_{-w}.
\]
By construction, we thus have
\begin{equation}
\label{costruita}
\mathfrak{F}_{q,\alpha}(\varphi_n)<\Lambda_1+\frac{1}{n+1},\qquad \|\varphi_n-w\|_{W^{1,2}_0(\Omega)}>\frac{\ell}{4},\qquad \|\varphi_n-(-w)\|_{W^{1,2}_0(\Omega)}>\frac{\ell}{4}.
\end{equation}
The first fact shows that $\{\varphi_n\}_{n\in\mathbb{N}}$ is a minimizing sequence for $\mathfrak{F}_{q,\alpha}$ and thus it converges strongly either to $w$ or to $-w$, again thanks to Lemma \ref{lm:minseq}. But this contradicts either the second or the third property in \eqref{costruita}.
\end{proof}

\section{Stabilization for a rescaled problem}
\label{sec:5}

As explained in the Introduction, in order to study the asymptotic profile of the unique weak solution $u$ to \eqref{mainprob}, one can reduce to study the long-time behavior of the time scaling transformation
\begin{equation}
v(x,t)=\text{e}^{\alpha\, t}\,u(x,e^{t}-1),\qquad \mbox{ where } \alpha=\frac{1}{m-1}.\label{scaling}
\end{equation}
It is easily seen that $v$ solves the  following problem
\begin{equation}\label{rescaledeq}
\left\{\begin{array}{rcll}
\partial_t v&=&\Delta \Phi(v)+\alpha\,v, & \mbox{ in } Q,\\
v&=&0, & \mbox{ on  } \Sigma,\\
v(\cdot,0)&=&u_0,& \mbox{ in } \Omega.
\end{array}
\right.
\end{equation}
The solution $v$ is understood in the weak sense: the equation in \emph{(ii)} in Definition \ref{weaksolutiPME} is replaced by
\[
\iint_{Q}\Big(\langle\nabla\Phi(v),\nabla\eta\rangle-v\,\partial_t \eta\Big)\,dx\,dt=\alpha\iint_{Q}v\, \eta\,dx\,dt,
\]
for any test function $\eta\in C_{0}^\infty(Q)$. Using the scaling transformation \eqref{scaling} and Theorem \ref{existence} it follows that there is a unique weak solution $v$ to \eqref{rescaledeq}.
\par
According to \cite[Chapter 20, Section 2]{VaBook}, an efficient method to study the asymptotic profiles of $v$ consists in adopting an abstract dynamical systems approach. This basically allows to see the solution $v$ as an orbit in some functional space and consider the so-called $\omega-$\emph{limit}, namely the points to which the solution itself accumulates as time goes to infinity.
\begin{defi}
We define the {\it $\omega-$limit set for a solution $v$ to \eqref{rescaledeq} emanating from the initial datum} $u_{0}$ as the set
\[
\omega(u_0)=\Big\{f\in L^{m+1}(\Omega)\, :\, \lim_{n\to+\infty} \|v(\cdot,t_n)-f\|_{L^{m+1}(\Omega)}=0 \mbox{ for some }\{t_n\}_{n\in\mathbb{N}}\nearrow +\infty\Big\}.
\]
\end{defi}
The choice of the functional space $L^{m+1}(\Omega)$ is justified by the regularity of the flow for all $t\geq0$, under our standing assumptions.
\vskip.2cm
The following result characterizes the $\omega-$limit, provided that the solution $v$ satisfies some a priori estimates.
It will be crucial for our main result. The result is due to Langlais and Phillips and is taken from \cite[Theorem 1.1]{LP}, except for the fact that we remove the global $L^\infty$ assumption on the solution. For this reason, we give the proof, which is an amended version of that contained in \cite{LP}.
\begin{teo}[Characterization of the $\omega-$limit]
\label{teo:LP}
Let $m>1$ and let $\Omega\subset\mathbb{R}^N$ be an open bounded set.
If $v$ is the unique weak solution of \eqref{rescaledeq}, we suppose that there exists $T_0>0$ such that
\begin{equation}
\label{ipotesiLP}
\partial_t g(v)\in L^2([T_0,+\infty);L^2(\Omega)) \qquad \mbox{ and }\qquad \nabla \Phi(v)\in L^\infty([T_0,+\infty);L^2(\Omega)),
\end{equation}
being $g$ the nonlinearity in \eqref{g}.
Then every $\psi\in \omega(u_0)$ is such that $\Phi(\psi)\in W^{1,2}_0(\Omega)$ and it is a weak solution of the Lane-Emden equation \eqref{LE}, with
\[
q=\frac{m+1}{m}.
\]
\end{teo}
\begin{proof}
Let us take $\psi\in \omega(u_0)$, then there exists a diverging sequence of times $\{t_n\}_{n\in\mathbb{N}}$ such that
\[
\lim_{n\to\infty} \|\psi-v(\cdot,t_n)\|_{L^{m+1}(\Omega)}=0.
\]
We first prove that $\Phi(\psi)\in W^{1,2}_0(\Omega)$. By using the elementary inequality \eqref{egolo0}, we have
\[
|\Phi(\psi(x))-\Phi(v(x,t_n))|^\frac{m+1}{m}\le C\,\Big(|\psi(x)|^{m-1}+|v(x,t_n)|^{m-1}\Big)^\frac{m+1}{m}\, |\psi(x)-v(x,t_n)|^\frac{m+1}{m}.
\]
If we now apply H\"older's inequality with exponents $m$ and $m/(m-1)$, with some simple algebraic manipulations we get
\[
\begin{split}
\int_\Omega |\Phi(\psi(x))-\Phi(v(x,t_n))|^\frac{m+1}{m}\,dx&\le C\,\left(\int_\Omega |\psi|^{m+1}\,dx+\int_\Omega |v(x,t_n)|^{m+1}\,dx\right)^\frac{m-1}{m}\\
&\times\left(\int_\Omega |\psi(x)-v(x,t_n)|^{m+1}\,dx\right)^\frac{1}{m}.
\end{split}
\]
Since the sequence $\{v(\cdot,t_{n})\}_{n\in\mathbb{N}}$ is bounded in $L^{m+1}(\Omega)$, we thus conclude that
\begin{equation}
\label{Phiconv}
\lim_{n\to\infty} \int_\Omega |\Phi(\psi(x))-\Phi(v(x,t_n))|^\frac{m+1}{m}\,dx=0.
\end{equation}
Moreover, the uniform estimate on the $W^{1,2}_0(\Omega)$ norm given by the second property in \eqref{ipotesiLP} and the continuity property $v\in C([0,+\infty);L^{m+1}(\Omega))$ implies that there exists
a constant $C>0$ such that
\[
\int_{\Omega}|\nabla \Phi(v(x,t))|^{2}dx\leq C,
\]
for \emph{all} $t\ge T_0$. Therefore, we get that $\Phi(v(\cdot,t_n))$ weakly converges in $W^{1,2}_0(\Omega)$ (up to a subsequence). By the uniqueness of the limit, such a function must coincide with $\Phi(\psi)$. Since $W^{1,2}_0(\Omega)$ is weakly closed, we finally get
\[
\Phi(\psi)\in W^{1,2}_0(\Omega),
\]
as desired.
\vskip.2cm\noindent
With an argument similar to that leading to \eqref{Phiconv}, we can also get
\begin{equation}
\label{L2g}
\lim_{n\to\infty} \left\|g(\psi)-g(v(\cdot,t_n))\right\|_{L^2(\Omega)}=0.
\end{equation}
Indeed, it is sufficient to use again Lemma \ref{lm:potenzeconvesse}, this time with $\gamma=(m+1)/2$, so to get
\[
\int_\Omega |g(\psi(x))-g(v(x,t_n))|^2\,dx\le C\,\int_\Omega \left(|\psi(x)|^\frac{m-1}{2}+|v(x,t_n)|^\frac{m-1}{2}\right)^2\,|\psi(x)-v(x,t_n)|^2\,dx.
\]
We leave the details to the reader.
\par
We now set
\[
V_n(x,s)=v(x,t_n+s),\qquad \mbox{ for } x\in\Omega,\ s\in(-1,1)\ \mbox{Â and } n\ge 1,
\]
and claim that
\begin{equation}
\label{claimLP}
\lim_{n\to\infty} \|g(V_n)-g(\psi)\|_{L^2(\Omega\times(-1,1))}=0.
\end{equation}
Indeed, by basic Calculus and Jensen's inequality, for $s\in(-1,1)$ we have
\[
\begin{split}
\int_\Omega |g(V_n(x,s))-g(v(x,t_n))|^2\,dx&=\int_\Omega |g(v(x,s+t_n))-g(v(x,t_n))|^2\,dx\\
&\le 2\,\int_\Omega \left(\int^{t_n+1}_{t_n-1} |\partial_t g(v(x,t))|^2\,dt\right)\,dx.
\end{split}
\]
A further integration on $s\in(-1,1)$ gives
\[
\iint_{\Omega\times[-1,1]} |g(V_n(x,s))-g(v(x,t_n))|^2\,dx\,ds\le C\,\int_\Omega \left(\int^{t_n+1}_{t_n-1} |\partial_t g(v(x,t))|^2\,dt\right)\,dx.
\]
By using the first assumption in \eqref{ipotesiLP} and the absolute continuity of the Lebesgue integral, we get
\[
\lim_{n\to\infty} \int_\Omega \left(\int^{t_n+1}_{t_n-1} |\partial_t g(v(x,t))|^2\,dt\right)\,dx=0,
\]
and thus
\[
\lim_{n\to\infty}\iint_{\Omega\times(-1,1)} |g(V_n(x,s))-g(v(x,t_n))|^2\,dx\,ds=0.
\]
Then \eqref{claimLP} follows from the latter and \eqref{L2g}, by using the triangle inequality. In turn, from \eqref{L2g} we can now infer convergence of $V_n$ itself. Indeed, by \eqref{egolo} with $\gamma=(m+1)/2$ we have
\[
|V_n(x,s)-\psi(x)|^{m+1}\le C\,|g(V_n(x,s))-g(w(x))|^2.
\]
An integration in space-time now gives
\begin{equation}
\label{convspostata}
\lim_{n\to\infty} \|\psi-V_n\|_{L^{m+1}(\Omega\times(-1,1))}=0.
\end{equation}
We finally prove that $\Phi(\psi)$ is a weak solution of the Lane-Emden equation.
We take a cut-off function in time $\rho\in C^\infty_0((-1,1))$, such that
\[
\rho\ge 0\qquad \mbox{ and }\qquad \int_{-1}^1 \rho(s)\,ds=1.
\]
We also take $\varphi\in C^\infty_0(\Omega)$ and insert the test function $\rho(t-t_n)\,\varphi(x)$ in the weak formulation of our equation. We obtain
\[
\begin{split}
\int_{t_n-1}^{t_n+1}\int_\Omega v\,\rho'(t-t_n)\,\varphi\,dx\,dt&=\int_{t_n-1}^{t_n+1}\int_\Omega \langle \nabla \Phi(v),\nabla \varphi\rangle\,\rho(t-t_n)\,dx\,dt\\
&-\alpha\,\int_{t_n-1}^{t_n+1}\int_\Omega v\,\varphi\,\rho(t-t_n)\,dx\,dt.
\end{split}
\]
We make the change of variable $s=t-t_n$ in the time integral. By recalling the definition of $V_n$ and integrating by parts the terms containing the gradient of $V_n$, we get
\[
\begin{split}
\int_{-1}^{1}\int_\Omega V_n\,\rho'\,\varphi\,dx\,dt&=-\int_{-1}^{1}\int_\Omega \Phi(V_n)\,\Delta \varphi\,\rho\,dx\,dt-\alpha\,\int_{-1}^{1}\int_\Omega V_n\,\varphi\,\rho\,dx\,dt.
\end{split}
\]
By using \eqref{convspostata}, we can pass to the limit as $n$ goes to $\infty$ in the previous identity, so to get
\[
\int_{-1}^{1}\int_\Omega \psi\,\rho'\,\varphi\,dx\,dt=-\int_{-1}^{1}\int_\Omega \Phi(\psi)\,\Delta \varphi\,\rho\,dx\,dt-\alpha\,\int_{-1}^{1}\int_\Omega \psi\,\varphi\,\rho\,dx\,dt.
\]
If we now use that both $\psi$ and $\varphi$ depends only on the spatial variable, while $\rho$ is a function of time integrating at $1$ and with compact support in $(-1,1)$, we get
\[
0=-\int_\Omega \Phi(\psi)\,\Delta \varphi\,dx-\alpha\,\int_\Omega \psi\,\varphi\,dx.
\]
By recalling that $\Phi(\psi)\in W^{1,2}_0(\Omega)$ and observing that
\[
\psi=|\Phi(\psi)|^{q-2}\,\Phi(\psi),\qquad \mbox{ for } q=\frac{m+1}{m},
\]
the previous identity shows that $\Phi(\psi)$ weakly solves the claimed equation.
\end{proof}
In the next result we show that the crucial assumptions \eqref{ipotesiLP} are actually verified, in our setting. The idea of the proof heavily relies on the so called {\it Lyapunov method}, aimed at constructing an energy functional which is decreasing along the flow, satisfying then a suitable \emph{entropy-entropy dissipation} inequality which will be essential to prove the relative compactness of the orbit of our solution $v$. However, some care is needed, due to the possible lack of regularity of solutions.
\begin{prop}
\label{prop:stimecruciali}
Let $m>1$ and let $\alpha>0$. Let $\Omega\subset\mathbb{R}^N$ be an open bounded set. Finally, we take $v$ to be the unique weak solution to the rescaled problem \eqref{rescaledeq}.
\par
If $u_0\in L^{m+1}(\Omega)$ is such that $\Phi(u_0)\in W^{1,2}_0(\Omega)$, then
\begin{equation}
\label{LP2}
\partial_tg(v)\in L^2([0,+\infty);L^2(\Omega)) \qquad \mbox{ and }\qquad \nabla \Phi(v)\in L^\infty([0,+\infty);L^2(\Omega)),
\end{equation}
being $g$ the nonlinearity in \eqref{g}. Moreover, the following entropy-entropy dissipation inequality holds for every $T>0$
\begin{equation}
\label{LP1}
\mathfrak{F}_{\frac{m+1}{m},\alpha}(\Phi(v(\cdot,T)))+\frac{4\,m}{(m+1)^{2}}\,\iint_{Q_{T}}|\partial_{t}g(v)|^{2}\,dx\,dt
\le \mathfrak{F}_{\frac{m+1}{m},\alpha}(\Phi(u_0)).
\end{equation}
\end{prop}
\begin{proof}
We first present the heuristics behind the proof.
By recalling the definition \eqref{functionalF}, for every $\varphi\in L^{m+1}(\Omega)$ such that $\Phi(\varphi)\in W^{1,2}_0(\Omega)$, we set
\[
\mathcal{V}[\varphi]=\mathfrak{F}_{\frac{m+1}{m},\alpha}(\Phi(\varphi))=\frac{1}{2}\,\int_\Omega |\nabla \Phi(\varphi)|^2\,dx-\frac{\alpha\,m}{m+1}\,\int_\Omega |\Phi(\varphi)|^\frac{m+1}{m}\,dx.
\]
Using the argument in \cite[Section 2]{Va}, we know that \emph{formally} this is a Lyapunov function, i.\,e. it decreases along the solution $t\mapsto v(t)$ because of the {\it entropy-entropy dissipation identity}
\[
-\frac{d}{dt} \mathcal{V}[v(\cdot,t)]=\mathcal{I}[v(\cdot,t)].
\]
Here $\mathcal{I}(t)$ is the entropy dissipation defined by
\[
\mathcal{I}[v(\cdot,t)]=\frac{4\,m}{(m+1)^2}\,\int_{\Omega} \left|\partial_t g(v(x,t))\right|^2\,dx.
\]
This formally leads to the entropy-entropy dissipation identity
\begin{equation}
\mathcal{V}[v(\cdot,t)]+ \int_{0}^{t}\mathcal{I}[v(\cdot,\tau)]\,d\tau= \mathcal{V}[u_{0}],\qquad \mbox{ for every } t\ge 0,\label{entropyentropdiss}
\end{equation}
which implies \eqref{LP1}. In order to justify this estimate rigourously, we will however take a slightly different path: rather than proving directly that $\mathcal{V}$ is a Lyapunov functional, we will go through a regularization procedure, use the entropy-entropy dissipation identity for this regularized problem to obtain an entropy bound and then passing to the limit in the regularization parameter.
This will give directly the weaker information \eqref{LP1}, which is however enough for our purposes.
\vskip.2cm
We divide the proof into various steps, for ease of readability. In the first six steps, we will prove the result under the additional assumption that $\Omega$ is smooth. Then, in the last step, we will briefly explain how to remove this requirement.
\vskip.2cm\noindent
{\bf Step 1: a regularized problem.} For the time being, let us suppose that $\Omega$ has a $C^\infty$ boundary. We take a regular approximation $\Phi_{n}$ of the nonlinearity $\Phi$, which in particular eliminates the degeneracy at $v=0$. A practical choice is
\[
\Phi_{n}(s)=m\,\int_0^s \left(\frac{1}{n} +\tau^{2}\right)^{\frac{m-1}{2}}\,d\tau,\qquad \mbox{ for } s\in\mathbb{R},
\]
which converges to $\Phi$ locally uniformly on $\mathbb{R}$.
Notice that $\Phi_n'(s)>0$ for every $s\in\mathbb{R}$ and that by construction
\begin{equation}
\label{uBfin}
\Phi^{\prime}_{n}(s)=m\,\left(\frac{1}{n} +s^{2}\right)^{\frac{m-1}{2}}\geq m\,|s|^{m-1}=\Phi^{\prime}(s).
\end{equation}
Let $v_{n}$ be the solution to the regularized problem
\begin{equation}
\label{approxim}
\left\{\begin{array}{rcll}
u_t&=&\Delta \Phi_n(u)+\alpha\,u, & \mbox{ in } Q,\\
u&=&0, & \mbox{ on } \Sigma,\\
u(\cdot,0)&=&v_{0,n},& \mbox{ in } \Omega,
\end{array}
\right.
\end{equation}
where $v_{0,n}\in C^\infty_0(\Omega)$ is a smooth approximation of $v_{0}$, such that
\begin{equation}
\label{prima}
\lim_{n\to\infty} \|v_{0,n}-v_0\|_{L^{m+1}(\Omega)}=0,
\end{equation}
and
\begin{equation}
\label{seconda}
\lim_{n\to\infty}\|\Phi_n(v_{0,n})-\Phi(v_{0})\|_{W^{1,2}_0(\Omega)}=0.
\end{equation}
Let us construct such a sequence.
At first, we choose any sequence $\{f_n\}_{n\in\mathbb{N}}\subset C^\infty_0(\Omega)$ such that
\[
\lim_{n\to\infty}\|f_{n}-\Phi(v_{0})\|_{W^{1,2}_0(\Omega)}=0.
\]
This is possible, thanks to the very definition of $W^{1,2}_0(\Omega)$.
Then we set
\[
v_{0,n}=\Phi_{n}^{-1}(f_{n}),\qquad \mbox{ for every } n\in\mathbb{N}\setminus\{0\},
\]
and observe that we still have $v_{0,n}\in C^\infty_0(\Omega)$, thanks to the fact that $\Phi_{n}\in C^{\infty}(\mathbb{R})$ and
\[
\Phi_n'(s)\not=0 \mbox{ for every } s\in \mathbb{R}\qquad \mbox{ and }\qquad \Phi_n(0)=0.
\]
Thus we have \eqref{seconda}. In order to prove \eqref{prima},
by using the definition of $v_{0,n}$, the triangle inequality and Lemma \ref{lm:techno}, we can infer
\[
\begin{split}
|v_{0,n}-v_{0}|&\leq |\Phi_{n}^{-1}(f_{n})-\Phi_{n}^{-1}(\Phi(v_{0}))|+|\Phi_{n}^{-1}(\Phi(v_{0}))-\Phi_{n}^{-1}(\Phi_{n}(v_{0}))|\\
&\leq C\,\left(|f_{n}-\Phi(v_{0})|^\frac{1}{m}+|\Phi(v_{0})-\Phi_{n}(v_{0})|^\frac{1}{m}\right),
\end{split}
\]
for some constant $C=C(m)>0$.
This implies that
\begin{equation}\label{convm+1}
\int_{\Omega}|v_{0,n}-v_{0}|^{m+1}\,dx\leq C\left(\int_{\Omega}|f_{n}-\Phi(v_{0})|^\frac{m+1}{m}\,dx+\int_{\Omega}|\Phi(v_{0})-\Phi_{n}(v_{0})|^\frac{m+1}{m}\,dx\right),
\end{equation}
possibly for a different constant $C=C(m)>0$. We now observe that
\[
W^{1,2}_0(\Omega)\hookrightarrow L^\frac{m+1}{m}(\Omega),
\]
since $\Omega$ is bounded and
\[
\frac{m+1}{m}<2.
\]
Thus, from the strong convergence in $W^{1,2}_0(\Omega)$, we get
\begin{equation}
\label{1}
\lim_{n\to\infty} \int_{\Omega}|f_{n}-\Phi(v_{0})|^\frac{m+1}{m}\,dx=0,
\end{equation}
as well. On the other hand, by construction of $\Phi_n$ we know that
\[
\lim_{n\to\infty} \Phi_n(v_0(x))=\Phi(v_0(x)),\qquad \mbox{ for a.\,e. }x\in\Omega,
\]
and
\[
|\Phi_n(v_0)|^\frac{m+1}{m}\le C\,(1+|v_0|^m)^\frac{m+1}{m}\in L^1(\Omega).
\]
An application of the Dominated Convergence Theorem then gives
\begin{equation}
\label{2}
\lim_{n\to\infty} \int_{\Omega}|\Phi(v_{0})-\Phi_n(v_0)|^\frac{m+1}{m}\,dx=0.
\end{equation}
By using \eqref{1} and \eqref{2} in \eqref{convm+1}, we finally obtain \eqref{prima}, as desired.
\vskip.2cm\noindent
{\bf Step 2: energy inequality for the regularized problem.}
Now, classical Regularity Theory for parabolic equations ensures that the solution $v_{n}(x,t)$ is smooth (see \cite[Theorem 6.1, Chapter V]{LSU}). Multiplying the equation \eqref{approxim} by $\partial_{t}\,\Phi_{n}(v_{n})$ and integrating in the spatial variable over $\Omega$, we get
\begin{equation}
\label{testate}
\int_\Omega |\partial_t v_n|^2\,\Phi'_n(v_n)\,dx=\int_\Omega \Delta \Phi_n(v_n)\,\partial_t \Phi_n(v_n)\,dx+\alpha\,\int_\Omega v_n\,\partial_t\Phi_n(v_n)\,dx.
\end{equation}
If we now introduce the smooth convex function $F_n$ defined through
\[
F_{n}'(s)=s\,\Phi_{n}'(s),\qquad \mbox{ for every } s\in\mathbb{R},
\]
and use an integration by parts in the integral containing the Laplacian in \eqref{testate},
we easily find
\[
-\frac{d}{dt}\left[\frac{1}{2}\int_{\Omega}|\nabla \Phi_{n}(v_{n})|^{2}\,dx-\alpha\int_{\Omega}F_{n}(v_{n})\,dx\right]=\int_{\Omega}\Phi_{n}'(v_{n})\,|\partial_{t}v_{n}|^{2}dx.
\]
By integrating this identity in time, we obtain
\begin{equation}
\label{entropydissineq}
\iint_{Q_{t}}\Phi'_{n}(v_{n})\,|\partial_t v_{n}|^{2}\,dx\,d\tau+\mathcal{V}_{n}[v_{n}(\cdot,t)]=
\mathcal{V}_{n}[v_{0,n}],\qquad \mbox{ for every }n\ge 1\ \mbox{Â and }t\ge 0,
\end{equation}
where the functional $\mathcal{V}_n$ is defined by
\[
\mathcal{V}_{n}[\varphi]:=\frac{1}{2}\int_{\Omega}|\nabla \Phi_{n}(\varphi)|^{2}\,dx-\alpha\int_{\Omega}F_{n}(\varphi)\,dx,\qquad \mbox{ for } \Phi_{n}(\varphi)\in W^{1,2}_0(\Omega),
\]
and it is the \emph{approximated} Lyapunov functional.
\par
We now wish to pass to the limit as $n$ goes to $\infty$ in \eqref{entropydissineq}. Before proceeding further, we need to analyze some properties of $F_n$. As a primitive of $s\mapsto s\,\Phi_n'(s)$ we can take
\begin{equation}
\label{ubFn}
F_{n}(s)=m\,\int_{0}^{s}\tau\,\left(\frac{1}{n} +\tau^{2}\right)^{\frac{m-1}{2}}\,d\tau=\frac{m}{m+1}\,\left[\left(\frac{1}{n}+s^2\right)^\frac{m+1}{2}-\left(\frac{1}{n}\right)^\frac{m+1}{2}\right].
\end{equation}
Observe that $F_n$ converges to
\[
F(s)=\frac{m}{m+1}\,|s|^{m+1},
\]
locally uniformly on $\mathbb{R}$.
We observe that
\[
|F'_{n}(s)|\geq |F'(s)|,\ \mbox{ for every } s\in\mathbb{R} \qquad \mbox{ and }\qquad F_n(0)=F(0),
\]
then these two facts immediately imply
\[
F_n(s)\ge F(s),\qquad \mbox{ for every }s\in\mathbb{R}.
\]
\vskip.2cm\noindent
{\bf Step 3: taking the limit in \eqref{entropydissineq} -- RHS}. Observe that the gradient term in the right-hand side of \eqref{entropydissineq} easily passes to the limit, thanks to \eqref{seconda}. Let us check the second term. By the triangle inequality, we have
\begin{equation}
\label{convergFn}
\begin{split}
\left|\int_\Omega [F_n(v_{0,n})-F(v_0)]\,dx\right|&\le \int_{\Omega}|F_{n}(v_{0,n})-F_n(v_{0})|\,dx+\int_{\Omega}|F_n(v_0)-F(v_{0})|\,dx.\\
\end{split}
\end{equation}
By using Lemma \ref{lm:techno2}, we have
\[
\int_\Omega |F_{n}(v_{0,n})-F_n(v_{0})|\,dx\le m\,\int_\Omega \Big((1+|v_{0,n}|^2)^\frac{m}{2}+(1+|v_{0}|^2)^\frac{m}{2}\Big)\,|v_{0,n}-v_0|\,dx.
\]
By using H\"older's inequality with exponents $m+1$ and $(m+1)/m$ and recalling \eqref{prima}, we get
\[
\lim_{n\to\infty} \int_{\Omega}|F_n(v_{0,n})-F_n(v_{0})|\,dx=0.
\]
Moreover, by virtue of \eqref{ubFn} and using that $F_n\ge F$, we find
\[
\int_{\Omega}|F_{n}(v_{0})-F(v_{0})|\,dx\le\frac{m}{m+1}\int_{\Omega}\left[ \left(\frac{1}{n} +|v_{0}|^{2}\right)^{\frac{m+1}{2}}-|v_{0}|^{m+1}
\right]\,dx.
\]
The last integral tends to $0$, thanks to the Monotone Convergence Theorem. Thus from \eqref{convergFn} we get
\[
\lim_{n\to\infty} \int_\Omega F_n(v_{0,n})\,dx=\int_\Omega F(v_0)\,dx.
\]
This finally shows that
\begin{equation}
\label{RHS}
\lim_{n\to\infty} \mathcal{V}_n[v_{0,n}]=\mathcal{V}[v_0]=\frac{1}{2}\,\int_\Omega |\nabla \Phi(v_0)|^2\,dx-\frac{\alpha\,m}{m+1}\,\int_\Omega |\Phi(v_0)|^\frac{m+1}{m}\,dx.
\end{equation}
{\bf Step 4: some uniform estimates}. We still need to pass to the limit in the left-hand side of \eqref{entropydissineq}. This is more delicate and we will need some uniform estimates for the solutions $v_n$.
\par
At first, we prove that the functionals $V_{n}$ are equi-coercive on $W^{1,2}_0(\Omega)$, uniformly in time. Indeed, \eqref{ubFn} clearly gives
\[
\int_{\Omega}F_{n}(v_n)\,dx\leq C\,\int_{\Omega}|v_n|^{m+1}\,dx+C\,|\Omega|,
\]
for some constant $C=C(m)>0$. Moreover, by virtue of \eqref{uBfin}, we have for every $n\ge 1$
\[
\begin{split}
\frac{1}{2}\,\int_{\Omega}|\nabla \Phi_{n}(v_{n})|^{2}\,dx=\frac{1}{2}\,\int_{\Omega}|\Phi'_n(v_n)|^2\,|\nabla v_{n}|^{2}\,dx&\ge \frac{1}{2}\,\int_\Omega |\Phi'(v_n)|^2\,|\nabla v_{n}|^{2}\,dx,
\end{split}
\]
so that
\begin{equation}
\label{H1sci}
\frac{1}{2}\,\int_{\Omega}|\nabla \Phi_{n}(v_{n})|^{2}\,dx\ge \frac{1}{2}\,\int_{\Omega}|\nabla \Phi(v_n)|^{2}\,dx.
\end{equation}
By recalling that $|v_n|^{m+1}=|\Phi(v_n)|^{(m+1)/m}$, the last two estimates show that
\[
\mathcal{V}_n[v_n(\cdot,t)]\ge \frac{1}{2}\,\int_\Omega |\nabla \Phi(v_n)|^2\,dx-C\,\int_{\Omega}|\Phi(v_n)|^\frac{m+1}{m}\,dx-C\,|\Omega|.
\]
We can now apply Young's inequality as in the proof of \eqref{coercive}, so to end up with the coercivity estimate
\begin{equation}
\label{uniformissima}
\mathcal{V}_n[v_n(\cdot,t)]\geq \frac{1}{C_1}\,\int_{\Omega}|\nabla\Phi(v_n)|^{2}\,dx-C_2,
\end{equation}
for two constants $C_1,C_2>0$ depending on $N,m$ and $\Omega$, but neither on $n$ nor on $t$.
\par
Then this last inequality, together with \eqref{entropydissineq} and \eqref{RHS}, shows that
\begin{equation}
\|\nabla \Phi(v_{n}(\cdot,t))\|_{L^{2}(\Omega;\mathbb{R}^N)}\le C,\qquad \mbox{ for every } n\ge 1\ \mbox{ and }t> 0,\label{estgradn}
\end{equation}
for some universal constant $C>0$.
\par
Up to now, we discarded the contribution of the time derivative in the energy inequality
\eqref{entropydissineq}. It is time to call it into play. We observe that from \eqref{uniformissima} we get the universal bound
\[
\mathcal{V}_n[v_n(t)]\ge -C_2,
\]
which is independent both of $n$ and $t$. By using this in \eqref{entropydissineq} and recalling \eqref{uBfin}, for all $T>0$ we can infer
\begin{equation}
\label{stimatempo}
\begin{split}
\frac{4\,m}{(m+1)^{2}}\,\iint_{Q_{T}}|\partial_t g(v_n)|^{2}\,dx\,
dt&=\iint_{Q_{T}}\Phi'(v_{n})\,|\partial_{t}v_{n}|^{2}\,dx\,dt\\
&\le\iint_{Q_{T}}\Phi'_n(v_{n})\,|\partial_{t}v_{n}|^{2}\,dx\,dt\leq C,
\end{split}
\end{equation}
with $C$ depending on $N, m$ and $\Omega$ only.
Thus $\{\partial_{t}g(v_n)\}_{n\in\mathbb{N}}$ is bounded in $L^{2}(Q)$.
\par
This in turn implies that $\{g(v_n)\}_{n\in\mathbb{N}}$ can be regarded a sequence of $L^2(\Omega)-$valued {\it equi-continuous curves}. Indeed, for every $t,s\in[0,+\infty)$ we have by basic Calculus, Minkowski's inequality and H\"older's inequality
\[
\begin{split}
\|g(v_n(\cdot,t))-g(v_n(\cdot,s))\|_{L^2(\Omega)}&\le \int_t^s \|\partial_\tau g(v_n(\cdot,\tau))\|_{L^2(\Omega)}\,d\tau\\
&\le |t-s|^\frac{1}{2}\,\left(\int_t^s \|\partial_\tau g(v_n(\cdot,\tau))\|^2_{L^2(\Omega)}\,d\tau\right)^\frac{1}{2}\\
&\le C\,|t-s|^\frac{1}{2}.
\end{split}
\]
We are going to show that we can apply Ascoli-Arzel\`a Theorem to the family $\{g(v_n)\}_{n\in\mathbb{N}}\subset C([0,+\infty);L^2(\Omega))$, on every time interval $[0,T]$.
\par
 At this aim, we need a uniform regularity estimate for $x\mapsto g(v_n(x,t))$. This is done as follows: we first extend by zero outside $\Omega$ the function $x\mapsto v_n(x,t)$. Then from \eqref{estgradn}, for every $t\ge0$ and every $h\in\mathbb{R}^N\setminus\{0\}$ we have
\begin{equation}
\label{aletto!}
\begin{split}
C\ge \int_\Omega |\nabla \Phi(v_n(x,t))|^2\,dx&=\int_{\mathbb{R}^N}|\nabla \Phi(v_n(x,t))|^2\,dx\\
&\ge \int_{\mathbb{R}^N} \frac{|\Phi(v_n(x+h,t))-\Phi(v_n(x,t))|^2}{|h|^2}\,dx,
\end{split}
\end{equation}
where the last inequality is a classical fact from the theory of Sobolev spaces. We now use the elementary inequality
\[
\Big||A|^\frac{m-1}{2}\,A-|B|^\frac{m-1}{2}\,B\Big|\le C_m\, \Big||A|^{m-1}\,A-|B|^{m-1}\,B\Big|^\frac{m+1}{2\,m},
\]
which follows from \eqref{egolo} in Appendix \ref{app:A}  with the choices
\[
\gamma=\frac{2\,m}{m+1},\qquad a=|A|^\frac{m-1}{2}\,A,\qquad b=|B|^\frac{m-1}{2}\,B.
\]
By recalling the definitions of $\Phi$ and $g$, this in turn implies that for every $t\ge0$ and every $h\in\mathbb{R}^N\setminus\{0\}$
\[
|\Phi(v_n(x+h,t))-\Phi(v_n(x,t))|^2\ge |g(v_n(x+h),t)-g(v_n(x,t))|^\frac{4\,m}{m+1}.
\]
Thus from \eqref{aletto!} we obtain the following {\it uniform fractional differentiability estimate}
\[
\sup_{|h|>0} \int_{\mathbb{R}^N} \left|\frac{g(v_n(x+h),t)-g(v_n(x,t))}{|h|^\frac{m+1}{2\,m}}\right|^\frac{4\,m}{m+1}\,dx\le C,
\]
for some constant $C>0$, depending on $N,\Omega$ and $m$, only. Observe that
\[
\frac{4\,m}{m+1}>2\qquad \mbox{ and }\qquad \frac{m+1}{2\,m}<1,
\]
thanks to the choice of $m$. Then we can argue as in \cite[Theorem 2.7]{BLP}. We observe that
\[
\begin{split}
\int_\Omega |g(v_n(x,t))|^\frac{4\,m}{m+1}\,dx&=\int_\Omega |\Phi(v_n(x,t))|^2\,dx\\
&\le \frac{1}{\lambda_1(\Omega)}\,\int_\Omega |\nabla \Phi(v_n(x,t))|^2\,dx\le C,
\end{split}
\]
thanks to Poincar\'e inequality and \eqref{estgradn}, which means that the sequence $\{g(v_n(\cdot,t))\}_{n\in\mathbb{N}}$ is bounded in $L^{4m/(m+1)}(\Omega)$.
The last two uniform estimates are now enough to apply the classical Riesz-Fr\'{e}chet-Kolmogorov compactness
theorem and get that, in particular, for every $T>0$ and every $t\in[0,T]$ the set
\[
\{g(v_n(\cdot,t))\}_{n\in\mathbb{N}}\subset L^2(\Omega),
\]
is relatively compact in the norm topology.
\par
We can finally apply the Banach space--valued version of the Ascoli-Arzel\`a Theorem on $[0,T]$ (see \cite[Lemma 1]{Si}) and get existence of a function $h\in C([0,T];L^2(\Omega))$ such that (up to a subsequence)
\[
\lim_{n\to\infty} \Big\|g(v_n)-h\Big\|_{C([0,T];L^2(\Omega))}=0.
\]
Observe that by arbitrariness of $T$, we actually get that $h\in C([0,+\infty);L^2(\Omega))$.
We claim that
\begin{equation}
\label{identification}
h=g(v),
\end{equation}
where $v$ is the solution of our original initial boundary value problem. In order to show this, we set $\widetilde{v}=g^{-1}(h)$. Thanks to \eqref{egolo} again, we have
\[
|v_n-\widetilde{v}|\le C_m\,|g\circ v_n-h|^\frac{2}{m+1},
\]
then by raising to the power $m+1$ and integrating in space, we get
\[
\int_\Omega |v_n(x,t)-\widetilde{v}(x,t)|^{m+1}\,dx\le C\,\int_\Omega |g(v_n(x,t))-h(x,t)|^2\,dx,
\]
hence
\begin{equation}
\label{converginterm}
\lim_{n\to\infty} 	\Big\|v_n-\widetilde{v}\Big\|_{C([0,T];L^{m+1}(\Omega))}=0.
\end{equation}
In particular, we get that $\widetilde{v}\in C([0,+\infty);L^{m+1}(\Omega))$.
This uniform convergence in turn implies, by repeating a similar argument used in the proof of Theorem \ref{teo:LP} (see \eqref{Phiconv}), that we have
\[
\lim_{n\to\infty} 	\Big\|\Phi(v_n)-\Phi(\widetilde{v})\Big\|_{C([0,T];L^{(m+1)/m}(\Omega))}=0.
\]
By recalling the uniform Sobolev estimate \eqref{estgradn}, the previous convergence and using an integration by parts, for every vector field $\phi\in C^1_0(\Omega;\mathbb{R}^N)$ and every $t\in[0,T]$ we have
\begin{equation}
\label{bof}
\begin{split}
\int_\Omega \Phi(\widetilde{v}(x,t))\,\mathrm{div\,}\phi(x)\,dx&=\lim_{n\to\infty}\int_\Omega \Phi(v_n(x,t))\,\mathrm{div\,}\phi(x)\,dx\\
&=-\lim_{n\to\infty}\int_\Omega \langle \nabla\Phi(v_n(x,t)),\phi(x)\rangle\,dx\\
&\le C\,\|\phi\|_{L^2(\Omega;\mathbb{R}^N)}.
\end{split}
\end{equation}
This shows that for every $t\in[0,T]$ the function $x\mapsto\Phi(\widetilde{v}(x,t))$ has a weak gradient in $L^2(\Omega;\mathbb{R}^N)$. Once we have proved existence of the weak gradient, from \eqref{bof} we easily get that such a gradient is the weak limit of $\{\nabla \Phi(v_n(\cdot,t))\}_{n\in\mathbb{N}}$, for every $t\in[0,T]$.
\par
Finally, by the lower semicontinuity of the $L^2$ norm with respect to the weak convergence, we get from \eqref{estgradn}
\[
\int_\Omega |\nabla \Phi(\widetilde v(x,t))|^2\,dx\le C,\qquad \mbox{ for every } t\in[0,T],
\]
for some universal $C>0$ independent of $T$, while weak closedeness of the space $W^{1,2}_0(\Omega)$ implies
\[
\Phi(\widetilde{v}(\cdot,t))\in W^{1,2}_0(\Omega),\qquad \mbox{ for every }t\in[0,T].
\]
We now have all the convergences and regularity properties needed to pass to the limit in the weak formulation of \eqref{approxim}: this permits to show that $\widetilde{v}$ solves \eqref{rescaledeq} on every $Q_T$. Then uniqueness of the solution permits to conclude that \eqref{identification} holds true.
\vskip.2cm\noindent
{\bf Step 5: taking the limit in \eqref{entropydissineq} -- LHS.} We take into account the left-hand side of \eqref{entropydissineq}.
By using \eqref{converginterm} (recall that $\widetilde{v}=v$) and
proceeding as in \eqref{convergFn}, we can then prove that
\[
\lim_{n\to\infty} \int_\Omega F_n(v_n(x,t))\,dx=\frac{m}{m+1}\,\int_\Omega |v(x,t)|^{m+1}\,dx=\frac{m}{m+1}\,\int_\Omega |\Phi(v(x,t))|\,dx.
\]
Moreover, by using \eqref{H1sci} and the weak lower semicontinuity of the Dirichlet integral, we get
\[
\liminf_{n\to\infty} \int_{\Omega}|\nabla \Phi_{n}(v_{n})|^{2}\,dx\ge \liminf_{n\to\infty}\,\int_{\Omega}|\nabla \Phi(v_n)|^{2}\,dx\ge \int_\Omega |\nabla \Phi(v)|^2\,dx,
\]
The last two displays show that
\begin{equation}
\mathcal{V}[v(\cdot,t)]\le\liminf_{n\to\infty}\mathcal{V}_{n}[v_{n}(\cdot,t)].\label{lsemicV}
\end{equation}
\vskip.2cm\noindent
{\bf Step 6: proof of \eqref{LP2}\,and \eqref{LP1}.}
The uniform estimate
\[
\sup_{t\in[0,T]}\int_{\Omega}|\nabla \Phi(v(x,t))|^{2}\,dx\leq C,
\]
has already been proved in {\bf Step 4}. Since the constant $C$ is independent of $T$, this shows the second item in \eqref{LP2}.
\par
Still from {\bf Step 4}, we know that
\[
\lim_{n\to\infty} \left\|g(v_n(\cdot,t))-g(v(\cdot,t))\right\|_{C([0,T];L^2(\Omega))}=0.
\]
This easily implies that $\{\partial_t g(v_n)\}_{n\in\mathbb{N}}$ weakly converges in $L^2(Q_T)$ to $\partial_t g(v)$.
From the weak lower semicontinuity of the $L^{2}$ norm and \eqref{stimatempo}, we can take this uniform estimate to the limit and get
\[
\frac{4\,m}{(m+1)^2}\,\iint_{Q_{T}}|\partial_t g(v)|^{2}\,dx\,
dt\leq \liminf_{n\to\infty}\iint_{Q_{T}}\Phi^{\prime}_n(v_{n})\,|\partial_{t}v_{n}|^{2}\,dx\,dt\leq C.
\]
This is the first item in \eqref{LP2}.
\par
Finally, the last formula in display, \eqref{lsemicV} and \eqref{RHS} imply that we can pass to the limit in \eqref{entropydissineq} and obtain the entropy-entropy dissipation inequality \eqref{LP1}. This concludes the proof for an open bounded set with smooth boundary.
\vskip.2cm\noindent
{\bf Step 7: removing the smoothness of $\partial\Omega$}. We assume now that $\Omega\subset\mathbb{R}^N$ is any open bounded set and argue as in the proof of  \cite[Theorem 5.7]{VaBook}. By \cite[page 319]{Ke}, $\Omega$ can be exhausted by an increasing sequence of smooth domains $\{\Omega_n\}_{n\in\mathbb{N}}$, i.\,e.
\[
\Omega_n\Subset\Omega_{n+1}\qquad \mbox{ and }\qquad \Omega=\bigcup_{n\in\mathbb{N}} \Omega_n.
\]
We then choose an increasing sequence of cut-off functions $\xi_{n}\in C_{0}^{\infty}(\Omega_{n+1})$, such that
\[
\xi_n\equiv 1 \mbox{ on }\Omega_n,\qquad \mbox{ for every } n\in\mathbb{N}.
\]
By setting $\varphi_{0,n}:=v_{0}\,\xi_{n}$, let us  consider the solution $\varphi_{n}$ to equation \eqref{rescaledeq} on the space-time cylinder $Q_{n+1}:=\Omega_{n+1}\times(0,+\infty)$, with zero boundary data on $\Sigma_{n+1}:=\partial\Omega_{n+1}\times[0,+\infty)$ and initial datum $\varphi_{0,n}$ on $\Omega_{n+1}$. By the previous step, we know that the entropy-entropy dissipation inequality \eqref{entropyentropdiss} holds for each $\varphi_n$. Up to extending $\varphi_n$ to $0$ in $Q\setminus Q_{n+1}$, we have that all the integrals in the space variable can be considered in the whole $\Omega$.
\par
The same arguments used in the previous steps shows that $\Phi(v_{n})$ converges weakly (up to a subsequence) to $\Phi(v)$, where $v$ is the actual solution to problem \eqref{rescaledeq}. Moreover, by lower semicontinuity, we can take \eqref{entropyentropdiss} to the limit. We leave the details to the reader.
\end{proof}

\begin{oss}
For every $h\in\mathbb{R}^N$ and every measurable function $\psi:\mathbb{R}^N\to\mathbb{R}$, we set
\[
\delta_h \psi(x)=\psi(x+h)-\psi(x).
\]
Then, in passing, we notice that as a consequence of the estimates in {\bf Step 4} above, we obtain that
\[
\sup_{t\in[0,+\infty)}\sup_{|h|>0} \int_{\mathbb{R}^N} \left|\frac{\delta_h g(v(x,t))}{|h|^\beta}\right|^q\,dx<+\infty,
\]
with
\[
\beta=\frac{m+1}{2\,m}\qquad \mbox{ and }\qquad q=\frac{4\,m}{m+1}.
\]
This can be regarded as a spatial regularity estimate on the scale of fractional Sobolev spaces, uniform in time. For finer higher differentiabilty results for the solutions of PME--type equations, we refer to \cite{Eb, GST}  and \cite{TT}.
\end{oss}

\section{Proof of the main results}
\label{sec:6}

\begin{proof}[Proof of Theorem \ref{teo:main}]
The assumption $\Phi(u_0)\in W^{1,2}_0(\Omega)$ entails that there exists a unique weak solution $u$, thanks to Theorem \ref{existence}. Moreover, we also have that
\[
u\in L^\infty([0,+\infty);L^{m+1}(\Omega)),\qquad \Phi(u)\in L^2([0,+\infty);W^{1,2}_0(\Omega)),
\]
and
\[
u\in C([0,+\infty);L^{m+1}(\Omega)).
\]
We define the rescaled function $v$ solving problem \eqref{rescaledeq}. Then we have
\[
v\in L^\infty_{\rm loc}([0,+\infty);L^{m+1}(\Omega)),\qquad \Phi(v)\in L^2_{\rm loc}([0,+\infty);W^{1,2}_0(\Omega)),
\]
and
\[
v\in C([0,+\infty);L^{m+1}(\Omega)).
\]
We introduce the shortcut notation
\[
q=\frac{m+1}{m}\in (1,2),
\]
as in the statement. Accordingly, we define the functional $\mathcal{V}$ as in the proof of Proposition \ref{prop:stimecruciali}
\[
\mathcal{V}[v(\cdot,t)]:=\mathfrak{F}_{q,\alpha}(\Phi(v(\cdot,t)))=\frac{1}{2}\,\int_\Omega |\nabla\Phi(v(x,t))|^2\,dx-\frac{\alpha\,m}{m+1}\,\int_\Omega |\Phi(v(x,t))|^\frac{m+1}{m}\,dx.
\]
By Proposition \ref{prop:stimecruciali} and the compact embedding $W_{0}^{1,2}(\Omega)\hookrightarrow L^{\frac{m+1}{m}}(\Omega)$, we get that the orbit
\[
\Big\{v(\cdot,t):\,t\geq0\Big\},
\]
is relatively compact in $L^{m+1}(\Omega)$. Thus,  by \cite[Theorems 1.4.5 \& 1.4.7]{Chueshov} the $\omega-$limit set $\omega(u_{0})$ is nonempty, compact and connected. Moreover, Proposition \ref{prop:stimecruciali} guarantees that we can apply Theorem \ref{teo:LP}.
Then, if $\psi\in \omega(u_0)$, we know that $\Phi(\psi)\in W^{1,2}_0(\Omega)$ and it weakly solves
\begin{equation}
\label{LEstrange}
-\Delta \Phi(\psi)=\alpha\,|\Phi(\psi)|^{q-2}\,\Phi(\psi),\qquad \mbox{ in }\Omega.
\end{equation}
We now want to estimate the energy of $\Phi(\psi)$. By definition of $\omega-$limit, we can infer existence of a diverging sequence of times $\{t_n\}_{n\in\mathbb{N}}$ such that
\[
\lim_{n\to\infty} \|v(\cdot,t_n)-\psi\|_{L^{m+1}(\Omega)}=0.
\]
We recall that this implies \eqref{Phiconv}, that is
\[
\lim_{n\to\infty} \|\Phi(v(\cdot,t_n))-\Phi(\psi)\|_{L^\frac{m+1}{m}(\Omega)}=0.
\]
 Thus, by lower semicontinuity\footnote{We use that the Dirichlet integral is weakly lower semicontinuous with respect to the strong $L^{(m+1)/m}(\Omega)$ convergence. Indeed, it is sufficient to write
\[
\begin{split}
\left(\int_\Omega |\nabla \varphi|^2\,dx\right)^\frac{1}{2}&=\sup_{\phi\in C^\infty_0(\Omega;\mathbb{R}^N)}\left\{ \int_\Omega \langle \nabla \varphi,\phi\rangle\,dx\, :\, \|\phi\|_{L^2(\Omega;\mathbb{R}^N)}\le 1\right\}\\
&=\sup_{\phi\in C^\infty_0(\Omega;\mathbb{R}^N)}\left\{ \int_\Omega \varphi\,\mathrm{div\,}\phi\,dx\, :\, \|\phi\|_{L^2(\Omega;\mathbb{R}^N)}\le 1\right\},
\end{split}
\]
and then observe that each $\varphi\mapsto \int_\Omega \varphi\,\mathrm{div\,}\phi\,dx$ is continuous with respect to the strong $L^{(m+1)/m}(\Omega)$ convergence.} we get
\[
\liminf_{n\to +\infty}\mathcal{V}[v(\cdot, t_n)]\ge\left[\frac{1}{2}\,\int_\Omega |\nabla \Phi(\psi)|^2\,dx-\frac{\alpha\,m}{m+1}\,\int_\Omega |\Phi(\psi)|^\frac{m+1}{m}\,dx\right]=\mathfrak{F}_{q,\alpha}(\Phi(\psi)).
\]
By Proposition \ref{prop:stimecruciali} again, we also know that
\[
\mathcal{V}[v(\cdot,t)]\le \mathcal{V}[u_0],\qquad \mbox{ for every }t\ge 0.
\]
The last two estimates and the assumption \eqref{condizione} finally entail that
\begin{equation}
\mathfrak{F}_{q,\alpha}(\Phi(\psi))<\Lambda_2.\label{minener}
\end{equation}
However, by \eqref{LEstrange} we have that $\Phi(\psi)$ is a critical point of $\mathfrak{F}_{q,\alpha}$. Thus \eqref{minener} and Proposition \ref{prop:dasolo} imply that $\Phi(\psi)$ is a minimizer of $\mathfrak{F}_{q,\alpha}$.
In view of Proposition \ref{prop:min}, we must have that
\[
\Phi(\psi)\in\{w,-w\},
\]
where $w$ and $-w$ are the unique minimizers of $\mathfrak{F}_{q,\alpha}$.
By using that the $\omega-$limit $\omega(u_{0})$ is a connected set, while $\{w,-w\}$ is obviously disconnected, we get the full convergence of $v(\cdot,t)$ either to $\Phi^{-1}(w)$ or to $\Phi^{-1}(-w)$.
\par
This in turn implies the claimed convergence of $t^\alpha\,u(\cdot,t)$ with respect to the $L^{m+1}(\Omega)$ strong topology. In order to upgrade this to a uniform convergence, it is now sufficient to reproduce the argument of \cite[Chapter 20, page 526]{VaBook}. We leave the details to the reader.
\end{proof}
\begin{oss}
As pointed out to us by an anonymous referee, the previous result implies that
$t\mapsto \Phi(v(\cdot,t))$ converges in the strong $W_{0}^{1,2}(\Omega)$ topology, as well.
\par
We briefly sketch the argument: we start with the following estimate\footnote{These can be justified by an approximation argument, as in the proof of Proposition \ref{prop:stimecruciali}: multiply both sides of the equation \eqref{approxim} by $\Phi_n(v_n)$, integrate over $(t,t+1)$ and then pass to the limit as $n$ goes to $\infty$.}, which is valid for every time instant $t>0$
\[
\begin{split}
\frac{1}{q}&\left(\int_{\Omega}|v(x,t)|^{m+1}dx-\int_{\Omega}|v(x,t+1)|^{m+1}dx\right)\ge \int_{t}^{t+1}\!\!\int_{\Omega}\Big[|\nabla\Phi(v(x,s))|^{2}-\alpha\,|v(x,s)|^{m+1}\Big]\,dx\,ds.
\end{split}
\]
This can be recast into
\begin{equation}
\label{w012converg}
\int_{t}^{t+1}\mathfrak{F}_{q,\alpha}(\Phi(v(\cdot,s)))\,ds\le \frac{1}{2\,q}\left(\int_{\Omega}|v(x,t)|^{m+1}dx-\int_{\Omega}|v(x,t+1)|^{m+1}dx\right)+\mathcal{R}(t),
\end{equation}
where
\[
\mathcal{R}(t)=\alpha\left(\frac{1}{2}-\frac{1}{q}\right)\int_{t}^{t+1}\int_{\Omega}|v(x,s)|^{m+1}dx\,ds.
\]
From the convergence of $v$ obtained in the proof of Theorem \ref{teo:main}, we get that the right-hand side of \eqref{w012converg} admits a limit as $t$ goes to $+\infty$, given by
\[
\alpha\,\left(\frac{1}{2}-\frac{1}{q}\right)\int_{\Omega}|w|^{m+1}dx=\mathfrak{F}_{q,\alpha}(w)=\Lambda_1.
\]
We also used \eqref{ecrit} and the minimality of $w$. This entails that we have
\[
\limsup_{t\to+\infty}\int_{t}^{t+1}\mathfrak{F}_{q,\alpha}(\Phi(v(\cdot,s)))\,ds\le \Lambda_1.
\]
On the other hand, still by minimality, we have
\[
\mathfrak{F}_{q,\alpha}(\Phi(v(\cdot,s)))\ge \Lambda_1,\qquad \mbox{ for every }s\ge 0,
\]
and thus we get existence of the limit for the left-hand side of \eqref{w012converg}.
As observed in the Introduction, we have that the map $t\rightarrow \mathfrak{F}_{q,\alpha}(\Phi(v(\cdot,t)))$ is decreasing: thus from the previous limit one also has
\[
\lim_{t\rightarrow+\infty}\mathfrak{F}_{q,\alpha}(\Phi(v(\cdot,t)))=\Lambda_1.
\]
An application of Lemma \ref{lm:minseq} now gives the desired conclusion.
\end{oss}

\begin{proof}[Proof of Proposition \ref{prop:seleziona}]
We use the notation of the proof of Theorem \ref{teo:main}. We argue by contradiction and assume that
\[
\lim_{t\to+\infty} \|t^\alpha u(\cdot,t)-\Phi^{-1}(-w)\|_{L^\infty(\Omega)}=0.
\]
By passing to the rescaled function $v$, this is the same as
\[
\lim_{t\to+\infty} \|v(\cdot,t)-\Phi^{-1}(-w)\|_{L^\infty(\Omega)}=0.
\]
By composing with the function $s\mapsto\Phi(s)$, we immediately have
\[
\lim_{t\to+\infty} \|\Phi(v(\cdot,t))-(-w)\|_{L^\infty(\Omega)}=0,
\]
as well.
Both the assumptions \eqref{vaialpositivo} and \eqref{rivaialpositivo} guarantee that
\[
\max\Big\{\mathfrak{F}_{q,\alpha}\left(\Phi(u_0^+)\right), \mathfrak{F}_{q,\alpha}\left(\Phi(u_0)\right)\Big\}<\Lambda_2.
\]
Thus if we apply  Lemma \ref{lm:mistero}, we can build a continuous path $\gamma:[0,1]\to W^{1,2}_0(\Omega)$ such that
\[
\gamma(0)=w,\qquad \gamma(1)=\Phi(u_0),\qquad \mathfrak{F}_{q,\alpha}(\gamma(t))<\Lambda_2,\ \mbox{ for every } t\in [0,1].
\]
Moreover, we have seen by the entropy-entropy dissipation inequality \eqref{LP1} that
\[
\mathfrak{F}_{q,\alpha}(\Phi(v(\cdot,t)))<\Lambda_2,\qquad \mbox{ for every } t\in[0,+\infty),
\]
and recall that $v\in C([0,+\infty);L^{m+1}(\Omega))$, with initial datum $u_0$. The uniform bound on the energy, together with \eqref{coercive} and the Sobolev-Poincar\'e inequality \eqref{sobolevpoincare}, yields the uniform estimate
\begin{equation}
\label{uniformeLm}
\int_\Omega |v(x,t)|^{m+1}\,dx=\int_\Omega \Big|\Phi(v(x,t))\Big|^\frac{m+1}{m}\,dx\le C,\qquad \mbox{ for every }t\ge 0.
\end{equation}
If we now use \eqref{egolo0},
we get
\[
\begin{split}
\int_\Omega& \Big|\Phi(v(x,t))-\Phi(v(x,s))\Big|^\frac{m+1}{m}\,dx\le C_m\int_\Omega (|v(t,x)|^{m-1}+|v(s,x)|^{m-1})^\frac{m+1}{m}\,|v(t,x)-v(s,x)|^\frac{m+1}{m}\,dx.
\end{split}
\]
By using H\"older's inequality with exponents $m$ and $m/(m-1)$ and the uniform bound \eqref{uniformeLm}, we get
\[
\Big\|\Phi(v(\cdot,t))-\Phi(v(\cdot,s))\Big\|_{L^\frac{m+1}{m}(\Omega)}\le C\,\|v(t,\cdot)-v(s,\cdot)\|^\frac{m+1}{m}_{L^{m+1}(\Omega)}.
\]
This finally gives that $\Phi(v)\in C([0,+\infty);L^\frac{m+1}{m}(\Omega))$.
\par
We can then define the continuous path
\[
\theta(t)=\left\{\begin{array}{ll}
\gamma(2\,t),& \mbox{ if } t\in \left[0,\dfrac{1}{2}\right],\\
&\\
\Phi(v(\cdot,2\,t-1)),& \mbox{ if } t\in \left[\dfrac{1}{2},+\infty\right),
\end{array}
\right.
\]
connecting $w$ and $-w$. Continuity here is intended with respect to the strong $L^\frac{m+1}{m}(\Omega)$ topology.
Thus, by construction, we have
\[
\theta\in \widetilde\Gamma_\infty=\Big\{\gamma\in C([0,+\infty);L^\frac{m+1}{m}(\Omega))\, :\, \mathrm{Im}(\gamma)\subset W^{1,2}_0(\Omega),\,\gamma(0)=w,\, \gamma(+\infty)=-w\Big\},
\]
and
\[
\sup_{\varphi\in \mathrm{Im}(\theta)}\mathfrak{F}_{q,\alpha}(\varphi)<\Lambda_2.
\]
Theorem \ref{teo:MP} now yields $\Lambda^*<\Lambda_2$,
which contradicts Proposition \ref{prop:nocollapse}.
\end{proof}

\appendix

\section{Modulus of continuity of some auxiliary functions}

\label{app:A}

\begin{lm}
\label{lm:potenzeconvesse}
Let $\gamma>1$, then for every $a,b\in\mathbb{R}$ we have
\begin{equation}
\label{egolo0}
\Big||a|^{\gamma-1}\,a-|b|^{\gamma-1}\,b\Big|\le \gamma\,\Big(|a|^{\gamma-1}+|b|^{\gamma-1}\Big)\,|a-b|,
\end{equation}
and
\begin{equation}
\label{egolo}
|a-b|\le 2^\frac{\gamma-1}{\gamma}\,\Big||a|^{\gamma-1}\,a-|b|^{\gamma-1}\,b\Big|^\frac{1}{\gamma}.
\end{equation}
\end{lm}
\begin{proof}
Inequality \eqref{egolo0} follows from the Mean Value Theorem. Indeed, we have (by assuming for simplicity that $b<a$)
\[
\Big||a|^{\gamma-1}\,a-|b|^{\gamma-1}\,b\Big|= \gamma\,|\xi|^{\gamma-1}\,|a-b|,
\]
for some $\xi\in [b,a]$. It is now sufficient to use that
\[
|\xi|^{\gamma-1}\le |a|^{\gamma-1}+|b|^{\gamma-1},\qquad \mbox{ for every } \xi\in [b,a].
\]
In order to prove \eqref{egolo}, we observe that the map
\[
\tau\mapsto|\tau|^{\frac{1}{\gamma}-1}\,\tau,
\]
is $1/\gamma-$H\"older continuous. More precisely, we have
\[
\left||\tau|^{\frac{1}{\gamma}-1}\,\tau-|\eta|^{\frac{1}{\gamma}-1}\,\eta\right|\le 2^\frac{\gamma-1}{\gamma}\,|\tau-\eta|^\frac{1}{\gamma}.
\]
We use this estimate with the choices
\[
\tau=|a|^{\gamma-1}\,a\qquad \mbox{ and }\qquad \eta=|b|^{\gamma-1}\,b,
\]
so to get
\[
|a-b|\le 2^\frac{\gamma-1}{\gamma}\,\Big||a|^{\gamma-1}\,a-|b|^{\gamma-1}\,b\Big|^\frac{1}{\gamma},
\]
as desired.
\end{proof}
For $\delta>0$ and $m>1$, we consider the monotone increasing function $\Phi_\delta:\mathbb{R}\to\mathbb{R}$ defined by
\[
\Phi_\delta(s)=m\,\int_0^s \left(\delta+\tau^{2}\right)^{\frac{m-1}{2}}\,d\tau,\qquad \mbox{ for } s\in\mathbb{R},
\]
which we used in Section \ref{sec:5}.
We indicate by $\Psi_\delta$ its inverse function, which is still monotone increasing.
\begin{lm}
\label{lm:techno}
Let $\delta>0$ and $m>1$, then we have
 \[
|\Psi_\delta(a)-\Psi_\delta(b)|\le 2^\frac{m-1}{m}\,|a-b|^\frac{1}{m},\qquad \mbox{ for every } a,b\in\mathbb{R}.
\]
\end{lm}
\begin{proof}
If $a=b$ there is nothing to prove. Without loss of generality, we can suppose that $a>b$, thanks to the monotonicity of $\Psi_\delta$. We first observe that
\[
\begin{split}
\Phi'_\delta(a)&=m\,(\delta+a^2)^\frac{m-1}{2}\ge m\,|a|^{m-1}.
\end{split}
\]
By using this estimate, for every $a>b$ we have
\[
|a|^{m-1}\,a-|b|^{m-1}\,b=m\,\int_b^a |\tau|^{m-1}\,d\tau\le \int_b^a \Phi'_\delta(\tau)\,d\tau=\Big(\Phi_\delta(a)-\Phi_\delta(b)\Big).
\]
By combining this with \eqref{egolo}, we obtain
\[
|a-b|\le 2^\frac{m-1}{m}\,|\Phi_\delta(a)-\Phi_\delta(b)|^\frac{1}{m}.
\]
If we now replace $a$ with $\Psi_\delta(a)$ and $b$ with $\Psi_\delta(b)$, the last estimate implies the desired result.
\end{proof}
\begin{lm}
\label{lm:techno2}
Let $\delta>0$ and $m>1$, we set
\[
F_\delta(s)=\int_0^s \tau\,\Phi'_\delta(\tau)\,d\tau=\frac{m}{m+1}\,(\delta+s^2)^\frac{m+1}{2}-\frac{m}{m+1}\,\delta^\frac{m+1}{2},\qquad \mbox{ for } s\in\mathbb{R}.
\]
Then we have
\[
|F_\delta(a)-F_\delta(b)|\le m\,\Big((\delta+a^2)^\frac{m}{2}+(\delta+b^2)^\frac{m}{2}\Big)\,|a-b|,\qquad \mbox{ for every } a,b\in\mathbb{R}.
\]
\end{lm}
\begin{proof}
For $a=b$, there is nothing to prove. Without loss of generality, we can assume that $a>b$. By the Mean Value Theorem, there exists $\xi\in[b,a]$ such that
\[
|F_\delta(a)-F_\delta(b)|=|\xi\,\Phi'_\delta(\xi)|\,|a-b|.
\]
We now observe that the function $s\mapsto t\,\Phi'_\delta(s)$ is monotone increasing, thus we get
\[
b\,\Phi_\delta'(b)\le \xi\,\Phi'_\delta(\xi)\le a\,\Phi'_\delta(b),\qquad \mbox{ for every }b\le \xi\le a.
\]
This in turn entails that
\[
|\xi\,\Phi'_\delta(\xi)|\le |b\,\Phi'_\delta(b)|+|a\,\Phi'_\delta(a)|,\qquad \mbox{ for every }b\le \xi\le a.
\]
Moreover, we have
\[
|a\,\Phi_\delta'(a)|\le m\,(\delta+a^2)^\frac{m}{2}.
\]
By using the last two inequalities in the initial identity, we get the desired conclusion.
\end{proof}

\section{Continuous paths with controlled energy}

The following technical result permits to construct particular paths with controlled energy. We used this for the proof of Proposition \ref{prop:seleziona}. We still denote by $\mathfrak{F}_{q,\alpha}$ the functional \eqref{functionalF}, whose unique positive minimizer has been denoted by $w$.
\begin{lm}
\label{lm:mistero}
Let $1<q<2$ and $\alpha>0$. Let $\Omega\subset\mathbb{R}^N$ be an open bounded connected set. For every $\varphi\in W^{1,2}_0(\Omega)$, there exists a continuous path $\gamma:[0,1]\to W^{1,2}_0(\Omega)$ such that
\begin{enumerate}
\item $\gamma(0)=w$ and $\gamma(1)=\varphi$;
\vskip.2cm
\item $\mathfrak{F}_{q,\alpha}(\gamma(t))\le \max\Big\{\mathfrak{F}_{q,\alpha}(\varphi^+),\, \mathfrak{F}_{q,\alpha}(\varphi)\Big\}$, for every $t\in[0,1]$.
\end{enumerate}
\end{lm}
\begin{proof}
We first connect $\varphi^+$ to the global minimizer $w$, by controlling the energy along the path. We will take advantage of a subtle convex structure, hidden in our energy functional. We take the peculiar curve
\[
\sigma(t)=\Big((1-t)\,w^q+t\,(\varphi^+)^q\Big)^\frac{1}{q},\qquad t\in[0,1].
\]
By construction, we have
\[
\int_\Omega |\sigma(t)|^q\,dx=(1-t)\,\int_\Omega |w|^q\,dx+t\,\int_\Omega |\varphi^+|^q\,dx,
\]
while by \cite[Proposition 4]{Ka} (see also \cite[Proposition 2.6]{BFK}) we know that
\[
\int_\Omega |\nabla\sigma(t)|^2\,dx\le (1-t)\,\int_\Omega |\nabla w|^2\,dx+t\,\int_\Omega |\nabla \varphi^+|^2\,dx.
\]
These entail that
\begin{equation}
\label{mozzo0}
\mathfrak{F}_{q,\alpha}(\sigma(t))\le (1-t)\,\mathfrak{F}_{q,\alpha}(w)+t\,\mathfrak{F}_{q,\alpha}(\varphi^+)\le \mathfrak{F}_{q,\alpha}(\varphi^+),
\end{equation}
where in the second inequality we used the minimality of $w$. If $\varphi\ge 0$ almost everywhere in $\Omega$, then $\varphi=\varphi^+$ and the proof is over.
\par
Otherwise, we consider the continuous path
\[
\eta^+(t)=\varphi^+-t\,\varphi^-,\qquad \mbox{ for } t\in[0,1],
\]
which connects $\varphi^+$ to $\varphi$. Let us compute the energy along this path: we have
\begin{equation}
\begin{split}
\mathfrak{F}_{q,\alpha}(\eta^+(t))&=\frac{1}{2}\,\int_\Omega |\nabla \varphi^+|^2\,dx-\frac{\alpha}{q}\,\int_\Omega |\varphi^+|^q\,dx\nonumber\\
&+\frac{t^2}{2}\,\int_\Omega |\nabla \varphi^-|^2\,dx-\frac{\alpha\,t^q}{q}\,\int_\Omega |\varphi^-|^q\,dx.\label{splitposneg}
\end{split}
\end{equation}
We claim that
\begin{equation}
\label{mozzo}
\mathfrak{F}_{q,\alpha}(\eta^+(t))\le \max\Big\{\mathfrak{F}_{q,\alpha}(\varphi^+),\, \mathfrak{F}_{q,\alpha}(\varphi)\Big\},\qquad \mbox{ for every } t\in[0,1].
\end{equation}
At this aim, observe that the function
\[
h(t):=\frac{t^2}{2}\,\int_\Omega |\nabla \varphi^-|^2\,dx-\frac{\alpha\,t^q}{q}\,\int_\Omega |\varphi^-|^q\,dx,
\]
is such that
\[
h'(t)=t\,\int_\Omega |\nabla \varphi^-|^2\,dx-\alpha\,t^{q-1}\,\int_\Omega |\varphi^-|^q\,dx.
\]
Thus $h$ is increasing for $t\in[t_0,1]$ and decreasing for $t\in[0,t_0]$, where
\[
t_0=\left(\alpha\,\frac{\displaystyle\int_\Omega |\varphi^-|^q\,dx}{\displaystyle\int_\Omega |\nabla \varphi^-|^2\,dx}\right)^\frac{1}{2-q}.
\]
We now distinguish two cases: either $0<t_0<1$ or $t_0\ge 1$.
\par
In the case $0<t_0<1$, then we have
\[
h(t)\le \max\{h(0),h(1)\}=\max\{0,\mathfrak{F}_{q,\alpha}(\varphi^-)\}.
\]
By recalling that
\[
\mathfrak{F}_{q,\alpha}(\eta^+(t))=\mathfrak{F}_{q,\alpha}(\varphi^+)+h(t),
\]
this bound in turn implies that
\[
\begin{split}
\mathfrak{F}_{q,\alpha}(\eta^+(t))&\le \mathfrak{F}_{q,\alpha}(\varphi^+)+ \max\{0,\mathfrak{F}_{q,\alpha}(\varphi^-)\}\\
&=\max\Big\{\mathfrak{F}_{q,\alpha}(\varphi^+),\, \mathfrak{F}_{q,\alpha}(\varphi)\Big\},\qquad \mbox{ for every } t\in[0,1],
\end{split}
\]
which proves \eqref{mozzo}.
\par
In the case $t_0\ge 1$, then $h$ is monotone decreasing on the interval $[0,1]$ and thus
\[
h(t)\le h(0)=0.
\]
In this case we obtain
\[
\mathfrak{F}_{q,\alpha}(\eta^+(t))\le \mathfrak{F}_{q,\alpha}(\varphi^+)=\max\Big\{\mathfrak{F}_{q,\alpha}(\varphi^+),\, \mathfrak{F}_{q,\alpha}(\varphi)\Big\},\qquad \mbox{ for every } t\in[0,1],
\]
thus proving \eqref{mozzo} in this case, as well.
\par
By gluing together the two paths $\sigma$ and $\eta^+$ and using \eqref{mozzo0} and \eqref{mozzo}, we then obtain the desired conclusion.
\end{proof}

\end{document}